\documentclass[a4paper]{amsart}

\usepackage{tikz,mathrsfs,MnSymbol,wasysym}
\usetikzlibrary{cd,calc}
\tikzset{>=to}
\tikzcdset{arrow style=tikz}
\newcommand{\IndexC}[2][d]{\leftsup{\circlearrowleft}{\mathbf{I}}^{#1}_{#2}}
\newcommand{\leftsup}[2]{{\vphantom{#2}}^{#1}{#2}}

\newtheorem{theorem}{Theorem}
\newtheorem{lemma}[theorem]{Lemma}
\newtheorem{proposition}[theorem]{Proposition}
\newtheorem{corollary}[theorem]{Corollary}
\theoremstyle{definition}
\newtheorem{definition}[theorem]{Definition}

\newtheorem{observation}[theorem]{Observation}
\newtheorem{notation}[theorem]{Notation}
\theoremstyle{remark}
\newtheorem{example}[theorem]{Example}
\newtheorem{remark}[theorem]{Remark}

\DeclareMathOperator{\End}{End}
\DeclareMathOperator{\Hom}{Hom}
\DeclareMathOperator{\Ext}{Ext}

\DeclareMathOperator{\Endo}{End} \renewcommand{\End}{\Endo}
\DeclareMathOperator{\modules}{mod} \renewcommand{\mod}{\modules}

\DeclareMathOperator{\add}{add}
\DeclareMathOperator{\ind}{ind}

\newcommand{\Db}{{\rm D^b}}

\newcommand{\D}{{\rm D}\!}
\newcommand{\iso}{\cong}

\newcommand{\sus}{\Sigma}
\newcommand{\mult}[2]{[#2\colon#1]}

\newcommand{\Gr}{\operatorname{K}_0^{\text{split}}}
\renewcommand{\index}{\operatorname{index}}
\newcommand{\sh}{\operatorname{sh}}
\newcommand{\vind}{\overrightarrow{\index}}
\newcommand{\sign}{\operatorname{sign}}
\title[Coefficients for higher cluster categories]{Tropical coefficient dynamics for higher-dimensional cluster categories}

\author[Oppermann]{Steffen Oppermann}\address[S. Oppermann]{Department of mathematical sciences, Norwegian University of Science and Technology}\email{steffen.oppermann@ntnu.no}
\author[Thomas]{Hugh Thomas}\address[H. Thomas]{Département de mathématiques, Université du Québec à Montréal}\email{thomas.hugh\_r@uqam.ca}

\begin{document}

\begin{abstract}
We show that the index in higher-dimensional cluster categories mutates according to a higher-dimensional version of tropical coefficient dynamics. 
\end{abstract}

\maketitle

\section{Introduction}

The simplest cluster algebras are those
which admit a categorification built from a 1-representation
finite (i.e., representation finite and hereditary) algebra.
There is a theory of $d$-representation finite algebras for $d>1$
defined in {\cite{IO}}.
This seems a promising place to go looking for ``higher'' analogues
of cluster phenomena.  

In a previous paper \cite{OT}, we focussed on the $(d-1)$-higher Auslander
algebra of linearly oriented $A_n$, written $A^d_n$, a particular example of $d$-representation finite algebras.
%The cluster tilting objects of the
%cluster category of linearly oriented $A_n$ are well-known to correspond to
%triangulations of an $n+3$-gon.  
We showed that the basic cluster tilting objects in $\mathscr C_n^{d}$,
the $d$-cluster tilting
category of $A_n^d$ correspond to triangulations of a $2d$-dimensional cyclic
polytope with $n+2d+1$ vertices. (This generalizes the well-known
description of the basic cluster tilting objects in $\mathscr C_n^1$ as corresponding
to triangulations of an $(n+3)$-gon.)
Mutation of cluster tilting objects corresponds to bistellar flip of
triangulations, the appropriate higher-dimensional generalization of
diagonal flips.
We did not find an analogue of a cluster algebra associated to this model.
However, we showed that a tropical version of the cluster algebra
exchange relations has a geometrically meaningful interpretation.

Within the network of ideas connected to cluster algebras,
as well as the cluster variable
dynamics ($X$-dynamics), there is also the system of coefficient dynamics ($Y$-dynamics), which can
likewise be tropicalized. It turns out that tropical coefficient dynamics play
an important role in the study of cluster algebras, because they govern the
evolution of $g$-vectors, which are used to define bases of cluster algebras.
The goal of this paper is to exhibit a
higher-dimensional version of tropical coefficient dynamics.  It turns out
that while some of our construction goes through for any positive $d$, we
get the best results when $d$ is odd. We therefore assume this for the remainder
of the introduction.  Note that for our main results, we do not
assume that our ambient category is associated to the higher Auslander algebra
$A^d_n$ --- our results apply to any 2-Calabi-Yau $(d+2)$-angulated
category, including $\mathscr C^d_n$ as a special case.

Let $\mathscr C$ be a 2-Calabi-Yau $(d+2)$-angulated category. 
Let $T$ be a
cluster tilting object.  Any object $X$ in $\mathscr C$ admits a resolution  whose
terms are in $\add T$, i.e., a $(d+2)$-angle of the form:  
$$ T_{d}\rightarrow T_{d-1} \rightarrow \dots \rightarrow T_0 \rightarrow X \rightarrow \Sigma T_d$$
with $T_i\in \add T$ for all $i$.  We define
$$\index_T(X)=\sum_{i=0}^{d} (-1)^i [T_i]$$
as an element of $\Gr(\add T)$ (i.e., the free abelian group on the indecomposable summands of $T$).

Now suppose that $T^*$ is a cluster tilting object obtained from $T$ by a
single mutation.  The main result of this paper shows how to calculate
$\ind_{T^*}(X)$ from $\ind_T(X)$ (using no other information about $X$).
%We defer the exact formula, since it requires some further definitions.
The result can be summarized in a phrase by saying that the index
satisfies a higher analogue of tropical coefficient dynamics.

\begin{theorem}
Let \( T \) be a cluster tilting object in a 2-Calabi-Yau \( (d+2)\)-angulated category, for some odd \(d\). Let \( T^* \) be obtained from \( T \) by a single mutation, that is by replacing some indecomposable summand \( E_m \) of \( T \) by a different indecomposable \( E_m^{\star} \). Assume there is no loop in the quiver of \( T \) at the vertex corresponding to \( E_m \).

\begin{itemize}
\item There are exchange \( (d+2) \)-angles
\[ \begin{tikzcd}[sep=small,cramped] \sus^{-1} E_m \ar[r] & E_m^{\star} \ar[r] & \overline{T}_d \ar[r] & \overline{T}_{d-1} \ar[r] & \cdots \ar[r] & \overline{T}_1 \ar[r] & E_m \end{tikzcd} \]
and
\[ \begin{tikzcd}[sep=small,cramped] E_m \ar[r] & \overline{T}^1 \ar[r] & \overline{T}^2 \ar[r] & \cdots \ar[r] & \overline{T}^d \ar[r] & E_m^{\star} \ar[r] & \sus E_m \end{tikzcd} \]
with the \( \overline{T}_i \) and \( \overline{T}_i \) in \( \add T / E_m \).
\item
For any rigid object \( X \), the index of \( X \) with respect to \( T^* \) is obtained from the index of \( X \) with respect to \( T \) by replacing \( [E_m] \) by
\begin{align*} - [E_m^{\star}]  - \sum_{i} (-1)^i [\overline{T}_i] & \qquad \textrm{ if the coefficient of $[E_m]$ is positive}\\
  - [E_m^{\star}]  - \sum_{i} (-1)^ i [\overline{T}^i] & \qquad \textrm{ if the coefficient of $[E_m]$ is negative.}
\end{align*}

\end{itemize}

\end{theorem}

In Section~\ref{coeff} we explain
the usual meaning of tropical coefficient dynamics, and the result of Dehy and
Keller \cite{DK} showing that, in a 2-Calabi-Yau triangulated category, the
index satisfies tropical coefficient dynamics.

In Section~\ref{sect.d+2-ang} we recall the definition of \( (d+2) \)-angulated categories, and the fact that we can construct a \( 2 \)-Calabi-Yau \( (d+2) \)-angulated cluster category for any \( d \)-representation finite algebra.

In Section~\ref{sect.mutation} we investigate the notion of mutation for cluster tilting objects in \( 2 \)-Calabi-Yau \( (d+2) \)-angulated categories. In particular we give a criterion for the mutability of indecomposable summands of cluster tilting objects --- this is a new phenomenon in higher dimension, as all summands are mutable in dimension \( 1 \).

In Section~\ref{sect.index} we prove our main theorem, giving an explicit formula for the dynamics of the index under mutation.

In Section~\ref{laminations}, we return to the setting of
  \cite{OT} apply the index to define a notion of higher shear
  coordinates for laminations in a cyclic polytope. We end with an example
of the mutation of indices worked out in this setting in Section \ref{ex}.

\subsection*{Acknowledgements}
The authors thank the Centre for Advanced Studies at the Norwegian Academy of
Science, NTNU, UNB, UQAM, and the Mathematisches Forschungsinstitut Oberwolfach for their
kind hospitality at various stages of the writing of this paper.
H.T. was partially supported by the Canada Research Chairs program and an NSERC Discovery Grant.
  
\section{Coefficient dynamics and the index}\label{coeff}
This section is motivational.  The reader who does not find
the prospect of a link to cluster algebraic considerations to be motivational
may safely skip it. 

\subsection{Coefficient dynamics and tropical coefficient dynamics} \label{2-1}

  We follow the presentation in \cite{CA4}, since it is convenient for our
  purposes.
Coefficient dynamics are defined in a semifield $(\mathscr F,\cdot,+)$.
By definition, $(\mathscr F,\cdot)$ is an abelian group, and $+$ is a
commutative, associative operation such that multiplication distributes
over addition. To orient the reader, one example of a semifield is the
positive real numbers equipped with the usual operations of multiplication
and addition.

A $Y$-seed is an $n\times n$ skew-symmetric matrix $B$ together with an
$n$-tuple $y = (y_1,\dots,y_n)$ of elements of $\mathscr F$.

For $1\leq m \leq n$, we define a mutation operation $\mu_m$ on $Y$-seeds.
$\mu_m(B)$ is the result of the usual mutation operation on $n\times n$ skew-symmetric matrices; we do not recall it here.  The operation on the coefficients
is given as follows:
$$ \mu^B_m(y)_i = \left\{ \begin{array}{ll}  y_m^{-1} & \textrm {if $i=m$,}\\
  y_i  (1+y_m^{-\sign(b_{mi})})^{-b_{mi}}
  & \textrm {otherwise.}\end{array}\right.$$
Here $\sign(x)$ is $1$, $0$, or $-1$ as $x$ is positive, zero, or negative. We include \( B \) in the notation for \( \mu^B_m(y) \) to record the dependence of this operation on the matrix \( B \).

We obtain tropical coefficient dynamics by specializing to a tropical
semifield. Let $\mathbb P= \mathbb Z^r$. Coordinatewise addition defines an
  abelian group structure on $\mathbb P$. Define an operation $\oplus$ on
  $\mathbb P$ by $$(a_1,\dots,a_r)\oplus (b_1,\dots,b_r)=(\min(a_1,b_1),\dots,\min(a_r,b_r))$$
  Then $(\mathbb P,+,\oplus)$ defines a semifield structure on $\mathbb P$.
  Such semifields are known as \emph{tropical semifields}. (Note that the
  operation $+$ on $\mathbb P$ plays the role of multiplication in the
  definition of a semifield, while $\oplus$ plays the role of addition.)
Specializing the coefficient dynamics to a tropical semifield,
we obtain tropical coefficient dynamics.

For a tuple $v = (v_1,\dots,v_n) \in \mathbb P^n$, the mutation $\mu^B_m$ is by definition
\[ \mu^B_m(v)_i = \left\{ \begin{array}{ll}  -v_m & \textrm {if $i=m$}\\
     v_i + [0 \oplus (- \sign(b_{mi}) v_m)] (-b_{mi})  & \textrm {otherwise,}
     \end{array}\right. \]
Note that for a vector \( x \), the term \( [ 0 \oplus x ] \) just describes the vector which one obtains from \( x \) by replacing all positive entries by zero. If we write $[x]_- = - [0 \oplus x]$, and similarly $[x]_+ = - [0 \oplus (-x)]$, then the above simplifies to
\begin{equation} \label{trop}
\mu^B_m(v)_i =\left\{ \begin{array}{ll}  -v_m & \textrm {if $i=m$}\\
     v_i+ [b_{mi}]_+[v_m]_+ - [b_{mi}]_-[v_m]_- & \textrm {otherwise.}
     \end{array}\right. \end{equation}
(Here, we use the notion $[\ ]_+,[\ ]_-$ both on vectors in $\mathbb Z^r$ and on scalars in \( \mathbb Z \).) 

\subsection{Indices and tropical coefficient dynamics}\label{2-2}
  Let $\mathscr C$ be a 2-Calabi-Yau
  category, and let $X$ be a rigid object in the category.  For $T$ any
  cluster tilting object in $\mathscr C$, 
  there is a
  triangle
  $$ T_1 \rightarrow T_0 \rightarrow X \rightarrow \Sigma T_1$$
  with $T_0, T_1\in \add T$.  
  We define $\ind_T(X)$, the index of $X$ with respect to $T$,
  to be $[T_0]-[T_1]$ in
  $\Gr(\add T)$.

  Let $T$ and $T'$ be cluster tilting objects related by a mutation. We will explain in this section the relationship between
    $\ind_T(X)$ and $\ind_{T'}(X)$, and then make the connection to tropical coefficient mutation.
      
      Let $T = E_1 \oplus \cdots \oplus E_n$ be a cluster tilting object in $\mathscr C$.
For \( E_m \in \{ E_1, \ldots, E_{n} \} \) there is a unique other indecomposable \( E_m^{\star} \) such that replacing \( E_m \) with \( E_m^{\star} \) gives a new cluster tilting object \( T^{\star} \). This object is given by the following two exchange triangles:
\begin{eqnarray*} %\label{eq} 
E_m \rightarrow \overline{T}^{\text{left}} \rightarrow E_m^*\rightarrow \Sigma E_m\\ %\label{eq2}
E_m^* \rightarrow \overline{T}^{\text{right}} \rightarrow E_m\rightarrow \Sigma E_m'
\end{eqnarray*}
with \( \overline{T}^{\text{left}} \) and \( \overline{T}^{\text{right}} \) in \( \add \bigoplus_{j \neq m} E_j \).

The following result explains how the index changes under mutation of cluster tilting objects.

  \begin
{theorem}[{\cite[Theorem 3]{DK}}] \label{DK-thm}
Let $T$ and $T^*$ be cluster 
tilting objects in $\mathscr C$
differing by a single mutation at $E_m$, as above.  
Let $X$ be a rigid object in $\mathscr C$.  
Then $\index_{T^*}(X)$ is obtained from \( \index_T(X) \) by substituting for $[E_m]$:

\begin{align*} - [E_m^*]  + [\overline{T}^{\textup{right}}] && \textrm{ if the coefficient of $[E_m]$ in $\index_T(X)$ is positive}\\
  - [E_m^*]  + [\overline{T}^{\textup{left}}] && \textrm{ if the coefficient of $[E_m]$ in $\index_T(X)$ is negative}
\end{align*}
\end{theorem}

  We now relate the previous theorem to tropical coefficient dynamics. 

First we turn the indices into integer vectors by setting
$$\vind_T(X)= ([\index_T(X):E_1],\dots,[\index_T(X):E_n]),$$
and similarly for $\vind_{T^*}$. (Note that the definition of
$\vind_T(X)$ and $\vind_{T^*}(X)$ depend
implicitly on the ordering of the summands of $T$ and $T^*$ which we have
already fixed.) With this convention the mutation rule of Theorem~\ref{DK-thm} turns into
\[ \vind_{T^*}(X) = \mu^T_m ( \vind_T(X) ), \] with
\begin{equation}\label{right-left-trop}
\mu^T_m(v)_i = \begin{cases} -v_m & \text{if } i = m \\ v_i + \mult{E_i}{\overline{T}^{\text{right}}} v_m & \text{if } i \neq m \text{ and } v_m \geq 0 \\ v_i + \mult{E_i}{\overline{T}^{\text{left}}} v_m & \text{if } i \neq m \text{ and } v_m \leq 0. \end{cases}
\end{equation}

Let us now assume that there are no loops or 2-cycles adjacent to \( E_m \) in the quiver of \( T \).
 Let \( B \) be the skew-symmetric matrix whose entry $b_{jk}$ is the number of arrows from $k$ to $j$, minus the number from $j$ to $k$.

 Under these hypotheses, we have
 \begin{equation*}\label{right-left}
 \overline T^{\text{right}}
 \simeq \bigoplus_{j=1}^n E_j^{\oplus [b_{mj}]_+}\qquad
 \overline T^{\text{left}} \simeq \bigoplus_{j=1}^n E_j^{\oplus [b_{mj}]_-},
 \end{equation*}
 and then \( \mu^T_m \) of Equation~\eqref{right-left-trop} and \( \mu^B_m \) of Equation~\eqref{trop} coincide.

\section{$(d+2)$-angulated cluster categories} \label{sect.d+2-ang}

Throughout, let $k$ be a field. All categories appearing will be $k$-categories with finite dimensional morphism spaces and splitting idempotents. In particular they have the Krull-Schmidt property. We denote by $\D = \Hom_{k}(-, k)$ the standard duality of $k$-vector spaces.

Recall that a triangulated category is an additive category, equipped with an automorphism $\Sigma$ (called \emph{suspension}), and a class of $3$-morphism sequences of the form
\[ X \to Y \to Z \to \Sigma X. \]
The sequences in this chosen class are referred to as \emph{triangles}. They are required to satisfy certain axioms. Vaguely, the triangles behave like short exact sequences, but provide a more symmetric language.

A $(d+2)$-angulated category is defined similarly to a triangulated category, just that its special sequences -- the $(d+2)$-angles -- mimic the behaviour of longer exact sequences.

\begin{definition}[{\cite{GKO}}]
A \emph{$(d+2)$-angulated category} is an additive category $\mathscr{D}$ with an automorphism $\Sigma$, and a class of $(d+2)$-morphism sequences of the form
\[ \begin{tikzcd}[sep=small,cramped] D_0 \ar[r,"f_0"] & \cdots \ar[r,"f_d"] & D_{d+1} \ar[r,"f_{d+1}"] & \Sigma D_0 \end{tikzcd}. \]
subject to the following four axioms. The sequences in this chosen class are referred to as \emph{(distinguished) $(d+2)$-angles}.
\begin{itemize}
\item[(D1)] 
\begin{itemize} 
\item Sums and summands of distinguished $(d+2)$-angles are distinguished $(d+2)$-angles again.
\item For any object $D \in \mathscr{D}$ the trivial $(d+2)$-angle
\[ \begin{tikzcd}[sep=small,cramped] D \ar[r,"\rm id"] & D \ar[r] & 0 \ar[r] & \cdots \ar[r] & 0 \ar[r] & \Sigma D \end{tikzcd} \]
is a distinguished $(d+2)$-angle.
\item Any morphism is the first morphism in some distinguished $(d+2)$-angle.
\end{itemize}
\item[(D2)] The class of distinguished $(d+2)$-angles is closed under rotation. That is, for a $(d+2)$-angle as above also
\[ \begin{tikzcd}[sep=scriptsize,cramped] D_1 \ar[r,"f_1"] & \cdots \ar[r,"f_d"] & D_{d+1} \ar[r,"f_{d+1}"] & \Sigma D_0 \ar[r,"(-1)^d \Sigma f_0"] &[6mm] \Sigma D_1 \end{tikzcd} \]
and
\[ \begin{tikzcd}[sep=scriptsize,cramped] \Sigma^{-1} D_{d+1} \ar[r,"(-1)^d \Sigma^{-1} f_{d+1}"] &[11mm] D_0 \ar[r,"f_0"] & \cdots \ar[r,"f_d"] & D_{d+1} \end{tikzcd} \]
are distinguished $(d+2)$-angles.
\item[(D3)]
Given the solid part of the following commutative diagram, where the rows are $(d+2)$-angles
\[ \begin{tikzcd}
D_0 \ar[r,"f_0"] \ar[d,"\varphi_0"] & D_1 \ar[r,"f_1"] \ar[d,"\varphi_1"] & D_2 \ar[r,"f_2"] \ar[d,dashed,"\varphi_2"] & \cdots \ar[r,"f_d"] & D_{d+1} \ar[r,"f_{d+1}"] \ar[d,dashed,"\varphi_{d+1}"] & \sus D_0 \ar[d,"\sus \varphi_0"] \\
E_0 \ar[r,"g_0"] & E_1 \ar[r,"g_1"] & E_2 \ar[r,"g_2"] & \cdots \ar[r,"g_d"] & E_{d+1} \ar[r,"g_{d+1}"] & \sus E_0
\end{tikzcd} \]
it is possible to find the dashed morphisms \( \varphi_2 \) to \( \varphi_{d+1} \) such that the entire diagram commutes.
\item[(D4)]
In (D3) it is possible to choose the dashed morphisms in such a way that also the cone of the vertical morphism 
\[ \begin{tikzcd}[sep=large,cramped,ampersand replacement=\&]
E_0 \oplus D_1 \ar[r,"{\left( \begin{smallmatrix} g_0 & \varphi_1 \\ 0 & - f_1 \end{smallmatrix} \right)}"] \& E_1 \oplus D_2 \ar[r,"{\left( \begin{smallmatrix} g_1 & \varphi_2 \\ 0 & - f_2 \end{smallmatrix} \right)}"] \& \cdots \ar[r,"{\left( \begin{smallmatrix} g_d & \varphi_{d+1} \\ 0 & - f_{d+1} \end{smallmatrix} \right)}"] \& E_{d+1} \oplus \Sigma D_0 \ar[r,"{\left( \begin{smallmatrix} g_{d+1} & \Sigma \varphi_0 \\ 0 & - \Sigma f_0 \end{smallmatrix} \right)}"] \& \Sigma E_0 \oplus \Sigma D_1
\end{tikzcd} \]
is a distinguished $(d+2)$-angle. (Such a choice of the dashed arrows is sometimes called \emph{good}).
\end{itemize}
\end{definition}

The following elementary properties of \( (d+1) \)-angulated categories are immediate generalizations of the same properties for triangulated categories, with the exact same proof.

\begin{lemma}
Any morphism in a distinguished \( (d+2) \)-angle is a weak kernel of the following morphism, and a weak cokernel of the previous morphism.

Any monomorphism or epimorphism in a \( (d+2) \)-angulated category splits. In particular any \( (d+2) \)-angle in which at least one term is \( 0 \) is contractible.
\end{lemma}

While we did \emph{not} require uniqueness for the \( (d+2) \)-angle starting with a given morphism in (D2), we do have the following weak uniqueness.

\begin{lemma} \label{lem.d-cone_unique}
Let \( f_0 \colon D_0 \to D_1 \) be a morphism in a \( (d+2) \)-angulated category \( \mathcal{D} \). If we have two completions to \( (d+2) \)-angles
\[ \begin{tikzcd}[sep=small,cramped] D_0 \ar[r,"f_0"] & D_1 \ar[r,"f_1^i"] & D_2^i \ar[r,"f_2^i"] & \cdots \ar[r,"f_d^i"] & D_{d+1}^i \ar[r,"f_{d+1}^i"] & \Sigma D_0 \end{tikzcd} \quad i \in \{1,2\}, \]
such that all \( f_j^i \) with \( 2 \leq j \leq d \) are radical morphisms, then the two \( (d+2) \)-angles are isomorphic.
\end{lemma}

\begin{proof}
By (D4) we can find morphisms \( a_2, \ldots, a_{d+1} \) such that
\[ \begin{tikzcd}[sep=large,cramped,ampersand replacement=\&]
D_0 \oplus D_1 \ar[r,"{\left( \begin{smallmatrix} f_0 & \operatorname{id} \\ 0 & - f_1^1 \end{smallmatrix} \right)}"] \& D_1 \oplus D_2^1 \ar[r,"{\left( \begin{smallmatrix} f_1^2 & a_2 \\ 0 & - f_2^1 \end{smallmatrix} \right)}"] \& \cdots \ar[r,"{\left( \begin{smallmatrix} f_d^2 & a_{d+1} \\ 0 & - f_{d+1}^1 \end{smallmatrix} \right)}"] \& D_{d+1}^2 \oplus \Sigma D_0 \ar[r,"{\left( \begin{smallmatrix} f_{d+1}^2 & \operatorname{id} \\ 0 & - \Sigma f_0^1 \end{smallmatrix} \right)}"] \& \Sigma D_0 \oplus \Sigma D_1
\end{tikzcd} \]
is a $(d+2)$-angle.
By the automorphisms \( \left( \begin{smallmatrix} \operatorname{id} & 0 \\ f_0 & \operatorname{id} \end{smallmatrix} \right) \), \( \left( \begin{smallmatrix} \operatorname{id} & 0 \\ f_1^1 & \operatorname{id} \end{smallmatrix} \right) \), and \( \left( \begin{smallmatrix} \operatorname{id} & 0 \\ f_{d+1}^2 & \operatorname{id} \end{smallmatrix} \right) \) of the first, second, and \( (d+2) \)-nd term, respectively, this \( (d+2) \)-angle is isomorphic to
\[ \begin{tikzcd}[sep=large,cramped,ampersand replacement=\&]
D_0 \oplus D_1 \ar[r,"{\left( \begin{smallmatrix} 0 & \operatorname{id} \\ 0 & 0 \end{smallmatrix} \right)}"] \& D_1 \oplus D_2^1 \ar[r,"{\left( \begin{smallmatrix} 0 & a_2 \\ 0 & - f_2^1 \end{smallmatrix} \right)}"] \& \cdots \ar[r,"{\left( \begin{smallmatrix} f_d^2 & a_{d+1} \\ 0 & 0 \end{smallmatrix} \right)}"] \& D_{d+1}^2 \oplus \Sigma D_0 \ar[r,"{\left( \begin{smallmatrix} 0 & \operatorname{id} \\ 0 & 0 \end{smallmatrix} \right)}"] \& \Sigma D_0 \oplus \Sigma D_1
\end{tikzcd}, \]
that is isomorphic to the direct sum of the two trivial \( (d+2) \)-angles \( D_1 \to D_1 \to 0 \to \cdots \to 0 \to \Sigma D_1 \) and \( D_0 \to 0 \to \cdots \to \Sigma D_0 \to \Sigma D_0 \) plus
\[ \begin{tikzcd}[sep=large,cramped,ampersand replacement=\&]
0 \ar[r] \& D_2^1 \ar[r,"{\left( \begin{smallmatrix} a_2 \\ - f_2^1 \end{smallmatrix} \right)}"] \& D_2^2 \oplus D_3^1 \ar[r,"{\left( \begin{smallmatrix} f_2^2 & a_3 \\ 0 & - f_3^1 \end{smallmatrix} \right)}"] \& \cdots \ar[r,"{\left( \begin{smallmatrix} f_d^2 & a_{d+1} \end{smallmatrix} \right)}"] \& D_{d+1}^2 \ar[r] \& 0
\end{tikzcd}. \]
By the first part of axiom (D1) this last sequence is also a \( (d+2) \)-angle, and since it contains a zero object it is contractible.  Let \( h_i \colon D_i^2 \oplus D_{i+1}^1 \to D_{i-1}^2 \oplus D_i^1 \) be a homotopy witnessing this contractability, that is
\[ \begin{pmatrix} 1_{D_i^2} & 0 \\ 0 & 1_{D_{i+1}^1} \end{pmatrix} =   \begin{pmatrix} h_{i+1}^{11} & h_{i+1}^{12} \\ h_{i+1}^{21} & h_{i+1}^{22} \end{pmatrix} \begin{pmatrix} f_i^2 & a_{i+1} \\ 0 & - f_{i+1}^1 \end{pmatrix} + \begin{pmatrix} f_{i-1}^2 & a_i \\ 0 & - f_i^1 \end{pmatrix} \begin{pmatrix} h_i^{11} & h_i^{12} \\ h_i^{21} & h_i^{22} \end{pmatrix} \]

By assumption all the \( f_j^i \) are radical morphisms, so \( h_{i+1}^{21} a_{i+1} \) and \( a_i h_i^{21} \) need to be invertible. Thus \( a_{i+1} \) is split mono and \( a_i \) is split epi. Since this holds for any \( i \) (with obvious simplifications at the end of the complex), it follows that all the \( a_i \) are actually isomorphisms.
\end{proof}

One situation that lends itself to applying the theory of $(d+2)$ angulated categories particularly nicely is $d$-dimensional representation theory in the sense of \cite{I_moSubcat, I_higherAuslander, IO}.

\begin{definition}
A subcategory \( \mathscr{X} \) of an abelian or triangulated category \( \mathscr{C} \) is called \emph{\(d\)-cluster tilting} if it is functorially finite, and
\begin{align*}
 \mathscr{X} & =  \{ C \in \mathscr{C} \mid \Ext^i(C, \mathscr{X}) = 0 \; \forall i \in \{1, \ldots, d-1\} \} \\
 & =  \{ C \in \mathscr{C} \mid \Ext^i(\mathscr{X}, C) = 0 \; \forall i \in \{1, \ldots, d-1\} \}.
\end{align*}
An object \( X \) is called \( d \)-cluster tilting if the subcategory \( \mathscr{X} = \add X \) is. (One may note that if \( \mathscr{C} \) has finite dimensional \( \Hom \)-spaces then the condition that \( \add X \) is functorially finite is automatic.)
\end{definition}

\begin{definition}
A finite dimensional algebra \( \Lambda \) is called \emph{$d$-representation finite} if its global dimension is at most \( d \), and the category \( \mod \Lambda \) admits a \( d \)-cluster tilting module \( M \).
\end{definition}

In particular the summands of $M$ do not have any extensions of degrees smaller than $d$, but they do admit a nice theory of $d$-extensions. See \cite{IO}. This leads to them admitting a $(d+2)$-angulated version of derived categories.

\begin{definition} \label{def.d-derived cat}
Let $\Lambda$ be a $d$-representation finite algebra, and $M$ as in the definition above. The \emph{\((d+2)\)-angulated derived category} of \( \Lambda \), denoted by \( \mathscr{D} \), is the additive category given as follows:
\begin{itemize}
\item The indecomposable objects of \( \mathscr{D} \) are symbols \( \Sigma^i X \), where $i \in \mathbb{Z}$ and $X$ is an indecomposable summand of $M$.
\item The morphisms of \( \mathscr{D} \) are given as
\[ \Hom_{\mathscr{D}}(\Sigma^i X, \Sigma^j Y) = \begin{cases} \Hom_{\Lambda}(X, Y) & i = j \\ \Ext_{\Lambda}^d(X, Y) & i+1 = j \\ 0 & \text{otherwise.} \end{cases} \] 
\end{itemize}
\end{definition}

This should be seen as a generalization of the description of derived categories of hereditary algebras (\(d=1\)): As in that case, any object is a sum of shifts of modules, and morphisms are given as $\Hom$ and $\Ext^d$ from the module category.

\begin{theorem} \label{thm.d-derived is d-ang}
Let $\Lambda$ be a $d$-representation finite algebra. Then its \( (d+2) \)-angulated derived category $\mathscr{D}$ is \((d+2)\)-angulated, with suspension \( \Sigma \) (i.e.\ sending the object \( \Sigma^i X \) to \( \Sigma^{i+1} X \)).
\end{theorem}

\begin{proof}
By \cite[Theorem~1.21]{I_higherAuslander}, the subcategory
\[ \add \{ \Sigma^{id}_{\text{triang}} M \mid i \in \mathbb{Z} \} \subseteq \Db(\mod \Lambda) \]
is \( d \)-cluster tilting, where \( \Sigma_{\text{triang}} \) denotes the suspension in the triangulated category \( \Db(\mod \Lambda) \). By construction this subcategory is closed under \( \Sigma^d_{\text{triang}} \). Thus, by \cite{GKO}, it is \( (d+2) \)-angulated.
\end{proof}

\begin{remark}
The suspension of the \( (d+2) \)-angulated category \( \mathscr{D} \) is given as \( \Sigma = \Sigma^d_{\text{triang}} \). In particular one needs to be careful to not confuse \( \Sigma \) with \( \Sigma_{\text{triang}} \), the latter of which doesn't even define an endofunctor of \( \mathscr{D} \).
\end{remark}

\begin{definition}
A \emph{Serre functor} on a $k$-category $\mathscr{C}$ is an endofunctor $S \colon \mathscr{C} \to \mathscr{C}$ giving rise to a functorial isomorphism
\[ \D \Hom_{\mathscr{C}}(X, Y) \iso \Hom_{\mathscr{C}}(Y, SX). \]
If a Serre functor exists, then it is unique up to natural isomorphism.
\end{definition}

\begin{example}
Let \( \Lambda \) be a finite dimensional algebra of finite global dimension. Then the functor \( - \otimes_{\Lambda}^{\mathbb{L}} \D \Lambda \) is a Serre functor of \( \Db(\mod \Lambda) \).

When realising the \( (d+2) \)-angulated derived category of a \( d \)-representation finite algebra \( \Lambda \) as a subcategory of its ordinary derived category as in the proof of Theorem~\ref{thm.d-derived is d-ang} above, then \( - \otimes_{\Lambda}^{\mathbb{L}} \D \Lambda \) restricts to an endofunctor of \( \mathscr{D} \). Thus, in particular, it also defines a Serre functor on \( \mathscr{D} \).
\end{example}

\begin{proposition} \label{prop.D_directed}
Assume \( \Lambda \) is a triangular \( d \)-representation finite algebra. Then its \( (d+2) \)-angulated derived category \( \mathscr{D} \) is directed. In particular, all non-zero endomorphisms of indecomposable objects are isomorphisms.
\end{proposition}

\begin{proof}
By \cite[Theorem~1.23 and the preceding discussion]{I_higherAuslander}, we know that
\[ \mathscr{D} = \add \{ \Sigma_{\text{triang}}^{-id}S^i \Lambda \mid i \in \mathbb{Z} \}  \subseteq \Db(\mod \Lambda) \] in its incarnation as a subcategory of the triangulated derived category, where \( S = - \otimes_{\Lambda}^{\mathbb{L}} \D \Lambda \) is the Serre functor.

For \( i > 0 \) we have that
\[ \Sigma_{\text{triang}}^{-id} S^i \Lambda = \Sigma_{\text{triang}}^{-id} S^{i-1} \D \Lambda \]
is concentrated in (homologically) positive degrees. Therefore
\[ \Hom_{\mathscr{D}}( \Lambda, \Sigma_{\text{triang}}^{-id} S^i \Lambda) = 0. \]
It follows that any cycle of maps between indecomposable objects in \( \mathscr{D} \) necessarily lies entirely inside one of the \( \add \Sigma_{\text{triang}}^{-id} S^i \Lambda \). Since \( \Sigma_{\text{triang}} \) and \( S \) are autoequivalences we may as well consider the case \( i = 0 \), that is ask about cycles in \( \add \Lambda \). These however don't exist by the assumption of \( \Lambda \) being triangular.
\end{proof}

\begin{remark}
We do not know any examples of non-triangular \( d \)-representation finite algebras. Thus, from the point of view of treating examples, the assumption that \( \Lambda \) is triangular is rather mild.
\end{remark}

\begin{definition}
A $(d+2)$-angulated category is called \emph{$2$-Calabi-Yau} if $\Sigma^2$ is a Serre functor.
\end{definition}

\begin{remark}
The notion of Calabi-Yau triangulated categories is classical. Here we extend this notion -- by using the same definition verbatim -- to $(d+2)$-angulated categories.
\end{remark}

The $(d+2)$-angulated categories we are most interested in in this paper are versions of cluster categories. Let us first recall this concept in the classical situation. Briefly, the idea is to force the derived category to become $2$-Calabi-Yau.

\begin{definition}
Let \( \mathscr{C} \) be an additive category, \( \mathtt{F} \colon \mathscr{C} \to \mathscr{C} \) an automorphism of \( \mathscr{C} \). Then the orbit category of \( \mathscr{C} \) with respect to \( \mathtt{F} \), which is denoted by \( \mathscr{C} / \mathtt{F} \), is given as follows:
\begin{itemize}
\item The objects of \( \mathscr{C} / \mathtt{F} \) are the same as the objects of \( \mathscr{C} \).
\item The morphisms of \( \mathscr{C} / \mathtt{F} \) are given by
\[ \Hom_{\mathscr{C} / \mathtt{F}} (C, D) = \coprod_{i \in \mathbb{Z}} \Hom_{\mathscr{C}}(C, \mathtt{F}^i D). \]
\end{itemize}
It follows that there is a natural functor \( \pi \colon \mathscr{C} \to \mathscr{C} / \mathtt{F} \). For this functor, one observes that \( \pi \circ \mathtt{F} \iso \pi \), and moreover \( \pi \) is universal amongst functors with this natural isomorphism.
\end{definition}

\begin{definition}[\cite{BMRRT}]
Let $H$ be a hereditary algebra. The \emph{cluster category} associated to $H$ is the orbit category $\mathscr{C} = \Db(\mod H) / \Sigma^{-2} S$, where $S = - \otimes_{H}^{\mathbb{L}} \D H$ is the Serre functor on $\Db(\mod \Lambda)$.
\end{definition}

\begin{theorem}[\cite{K}]
For a hereditary algebra \( H \), the cluster category \( \mathscr{C} \) as above is triangulated, and the canonical functor \( \pi \colon \Db(\mod \Lambda) \to \mathscr{C} \) is a triangle functor. 
\end{theorem}

The above definition of cluster category can easily be generalized to \(d\)-re\-pre\-sen\-ta\-tion finite algebras.

\begin{definition}[{\cite[Section~5]{OT}}] \label{def.cluster_cat}
Let \( \Lambda \) be \( d \)-representation finite, and \( \mathscr{D} \) be its \((d+2)\)-angulated derived category as in Definition~\ref{def.d-derived cat}.

The \emph{\((d+2)\)-angulated cluster category of \( \Lambda \)} is the orbit category
\[ \mathscr{C} = \mathscr{D} /  \Sigma^{-2} S, \]
where \( \Sigma \) denotes the \((d+2)\)-angulated suspension on \( \mathscr{D} \), and \( S = - \otimes_{\Lambda}^{\mathbb{L}} \D \Lambda \) is the Serre functor.
\end{definition}

This name is justified by the following result:

\begin{theorem}[{\cite[Section~5]{OT}}]
If \( \Lambda \) is \( d \)-representation finite, then the \( (d+2) \)-angulated cluster category is \(2\)-Calabi-Yau \( (d+2) \)-angulated. 
\end{theorem}

\begin{observation} \label{obs.fundamental_domain}
Let \( \Lambda \) be \(d\)-representation finite. The subset
\[ \add (M \oplus \sus \Lambda) \subseteq \mathscr{D} \] 
is a \emph{fundamental domain} of the \( (d+2) \)-angulated cluster category in the sense that the natural projection functor induces a bijection between isomorphism classes of objects in the fundamental domain and isomorphism classes of objects in the \( (d+2) \)-angulated cluster category.

For \( X \) and \( Y \) in our fundamental domain we observe that
\[ \Hom_{\mathscr{C}}(X, Y) = \Hom_{\mathscr{D}}(X, Y) \oplus \Hom_{\mathscr{D}}(X, \sus^2 S^{-1} Y). \]
\end{observation}

\begin{proposition} \label{prop.C_no_loops}
Let \( \Lambda \) be triangular \( d \)-representation finite. Let \( X \in \mathscr{C} \) be an indecomposable object in the \( (d+2) \)-angulated cluster category of \( \Lambda \). Then
\begin{itemize}
\item any non-zero endomorphism of \( X \) is an automorphism;
\item for any \( (d+2) \)-angle
\[ \begin{tikzcd}[sep=small,cramped] X \ar[r] & C_0 \ar[r] & C_1 \ar[r] & \cdots \ar[r] & C_d \ar[r] & \sus X \end{tikzcd} \]
with all morphisms in the radical of \( \mathscr{C} \), we have \( X \not\in \add \{ C_0, \ldots, C_d\} \). 
\end{itemize}
\end{proposition}

\begin{proof}
We may assume that \( X \) is indecomposable projective. Then any endomorphism of \( X \) in \( \mathscr{C} \) is induced by an endomorphism in \( \mathscr{D} \). In particular the first claim follows from Proposition~\ref{prop.D_directed}.

For the second claim, assume that \( C_0 \) also lies in our fundamental domain. It follows that the given map from \( X \) to \( C_0 \) in \( \mathscr{C} \) is the image of a map in \( \mathscr{D} \). This map may be completed to a \( (d+2) \)-angle in \( \mathscr{D} \), with all morphisms being radical morphisms. By the uniqueness of Lemma~\ref{lem.d-cone_unique} we observe that the given \( (d+2) \)-angle is isomorphic to the image of the \( (d+2) \)-angle in \( \mathscr{D} \). Now, by Proposition~\ref{prop.D_directed}, \( X \) cannot appear in any other terms of the \( (d+2) \)-angle in \( \mathscr{D} \). On the other hand we may observe that the entire \( (d+2) \)-angle in \( \mathscr{D} \) lies in the fundamental domain (since both \( X \) and \( \Sigma X \) do), so it cannot contain any \( (\Sigma^2S^{-1})^i X \) for \( i \neq 0 \). It follows that the image in \( \mathscr{C} \) does not contain any other copies of \( X \) either.

\end{proof}

\section{Mutation of cluster tilting objects} \label{sect.mutation}

Let \(\mathscr{C}\) be a $(d+2)$-angulated category, with suspension $\sus$. Assume $\mathscr{C}$ is $2$-Calabi-Yau.

\begin{definition}
An object $T \in \mathscr{C}$ is called \emph{cluster tilting} if
\begin{enumerate}
\item it is \emph{rigid}, that is $\Hom_{\mathscr{C}}(T, \sus T) = 0$, and
\item it has the \emph{resolving property}, that is for any object $K \in \mathscr{C}$ there is a $(d+2)$-angle
\[ \begin{tikzcd}[sep=small,cramped] \sus^{-1} K \ar[r] & T_d \ar[r] & T_{d-1} \ar[r] & \cdots \ar[r] & T_1 \ar[r] & T_0 \ar[r] & K \end{tikzcd} \]
with $T_i \in \add T$.
\end{enumerate}
\end{definition}

The basic idea of mutation is to replace an individual indecomposable summand of a cluster tilting object. Let us fix names for the cluster tilting object and the summand we want to replace, which we will use throughout this section.

\begin{notation} \label{not.mut}
Throughout, \( T \in \mathscr{C} \) denotes a basic cluster tilting object. We denote its indecomposable summands by \( E_i \), so that \( T = \bigoplus_{i=1}^n E_i \). We pick \( m \in \{1, \ldots, n \} \) and try to replace the indecomposable summand \( E_m \). For convenience, we write \( \overline{T} = \bigoplus_{\substack{i = 1 \\ i \neq m}}^n E_i \) for the sum of the remaining summands.
\end{notation}

\begin{definition} \label{def.mutation}
A \emph{mutation} of \(T\) at \(E_m\) is given by a different (up to isomorphism) basic cluster tilting object \( T^{\star} = \overline{T} \oplus E_m^{\star} \). That is, we replace the indecomposable summand \(E_m\) by a different object \( E_m^{\star} \) without destroying the cluster tilting property.

We will see below that \( E_m^{\star} \) is unique and indecomposable if it exists. However, in contrast to the classical case (\(d=1\)), there need not be such an \( E_m^{\star} \). If it does exist, we call \(E_m\) \emph{mutable}.
\end{definition}

\begin{theorem} \label{thm.find_replacement}
  Let \( T \), \(E_m\) and \(\overline{T}\) be as in Notation~\ref{not.mut} above. Assume that there is no loop at \( E_m \) in the quiver of \(T\). Let \(E_m^{\star} \not\iso E_m \) be a non-zero object,  such that \( \overline{T} \oplus E_m^{\star} \) is basic. Then the following are equivalent:
\begin{enumerate}
\item \( \overline{T} \oplus E_m^{\star} \) is cluster tilting. (So it is a mutation of \(T\) at \(E_m\).)
\item \( \overline{T} \oplus E_m^{\star} \) is rigid.
\item There is a \( (d+2) \)-angle
\[ \begin{tikzcd}[sep=small,cramped] \sus^{-1} E_m \ar[r] & E_m^{\star} \ar[r] & \overline{T}_d \ar[r] & \overline{T}_{d-1} \ar[r] & \cdots \ar[r] & \overline{T}_1 \ar[r,"f"] & E_m \end{tikzcd} \]
with $\overline{T}_i \in \add \overline{T}$ such that $f$ is a right $\overline{T}$-approximation.

Such a \( (d+2) \)-angle will be called a \emph{right exchange $(d+2)$-angle}.
\item There is a \( (d+2) \)-angle
\[ \begin{tikzcd}[sep=small,cramped] E_m \ar[r,"g"] & \overline{T}^1 \ar[r] & \overline{T}^2 \ar[r] & \cdots \ar[r] & \overline{T}^d \ar[r] & E_m^{\star} \ar[r] & \sus E_m \end{tikzcd} \]
with $\overline{T}^i \in \add \overline{T}$ such that $g$ is a left $\overline{T}$-approximation.

Such a \( (d+2) \)-angle will be called a \emph{left exchange $(d+2)$-angle}.
\end{enumerate}
In particular, (3) and (4) show an object \( E_m^{\star} \) satisfying these conditions is unique (by Lemma~\ref{lem.d-cone_unique}) and indecomposable, if it exists.

Moreover, if the above equivalent conditions are satisfied then
\begin{enumerate}
\setcounter{enumi}{4}
\item For any object \( K \in \mathscr{C} \) there is a resolving \( (d+2) \)-angle
\[ \begin{tikzcd}[sep=small,cramped] \sus^{-1} K \ar[r] & T_d \ar[r] & T_{d-1} \ar[r] & \cdots \ar[r] & T_1 \ar[r] & T_0 \ar[r] & K \end{tikzcd} \]
with \( T_i \in \add T \), such that ``middle terms'' \(T_{d-1}, \ldots, T_1 \) all lie in \( \add \overline{T} \).
\end{enumerate}
\end{theorem}

The theorem improves for the case of cluster categories of (certain) \(d\)-re\-pre\-sen\-ta\-tion finite algebras. In that case, we need not worry about the existence of loops, and all the above conditions become equivalent:

\begin{theorem} \label{thm.mutability_for_clustercat}
  Let \( \Lambda \) be a triangular \( d \)-representation finite algebra. Let \(T \in \mathscr{C}_{\Lambda} \) be cluster tilting in the \( (d+2) \)-angulated cluster category of \( \Lambda \), and \( E_m \) an indecomposable summand of \( T \). If Condition~(5) of Theorem~\ref{thm.find_replacement} holds, then there is $E_m^*\not\iso E_m$ indecomposable satisfying (1) to (4) in Theorem~\ref{thm.find_replacement}.
\end{theorem}

Theorem \ref{thm.mutability_for_clustercat} gives a criterion for the summand $E_m$ of $T$ to be mutable: in the setting of the theorem, $E_m$ is mutable if and only if $E_m$ only ever appears in the first or last term of a $T$-resolving $(d+2)$-angle  of any $K\in \mathscr{C}_{\Lambda}$.

In the setting of \cite{OT}, which, combinatorially speaking, is the setting where cluster tilting objects correspond to triangulations of even-dimensional cyclic polytopes, Williams \cite[Theorem 4.10]{W} has given a combinatorial criterion for mutability of summands of a cluster tilting object. It would be interesting to relate the two criteria.

The remainder of this section is devoted to the proofs of Theorems \ref{thm.find_replacement} and \ref{thm.mutability_for_clustercat}.

\begin{proof}[{Proof of Theorem~\ref{thm.find_replacement}}] \ 

\paragraph{(2) \( \Longrightarrow \) (3)}

Let
\[ \begin{tikzcd}[sep=small,cramped] E_m^{\star} \ar[r] & T_d \ar[r] & \cdots \ar[r] & T_1 \ar[r] & T_0 \ar[r] & \sus E_m^{\star} \end{tikzcd} \]
be a $(d+2)$-angle with $T_i \in \add T$ (this exists since $T$ is cluster tilting). We may choose this in such a way that all the maps $T_{i+1} \to T_i$ lie in the radical of $\mathscr{C}$, by splitting off all isomorphisms between the \(T_i\).

Since $\Hom_{\mathscr{C}}(\overline{T}, \sus E_m^{\star}) = 0$ we have $T_0 = (E_m)^n$ for some $n$. If \( T_0 \) was zero then the map \( E_m^{\star} \to T_d \) would be split mono, contradicting \( E_m^{\star} \not \in \add T \). Thus \( n \geq 1 \).

For $i > 0$ we decompose $T_i = \overline{T}_i \oplus (E_m)^{n_i}$, with $\overline{T}_i \in \add \overline{T}$. Applying $\Hom_{\mathscr{C}}(-, \sus E_m^{\star})$ to the $(d+2)$-angle above, we obtain the exact sequence
\begin{align*}
& \begin{tikzcd}[sep=small,cramped,ampersand replacement=\&] \End_{\mathscr{C}}(\sus E_m^{\star}) \ar[r] \& \Hom_{\mathscr{C}}(E_m, \sus E_m^{\star})^n \ar[r] \& \Hom_{\mathscr{C}}(E_m, \sus E_m^{\star})^{n_1} \ar[r] \& \cdots \end{tikzcd} \\
& \qquad  \begin{tikzcd}[sep=small,cramped,ampersand replacement=\&] \cdots \ar[r] \& \Hom_{\mathscr{C}}(E_m, \sus E_m^{\star})^{n_{d-1}} \ar[r] \& \Hom_{\mathscr{C}}(E_m, \sus E_m^{\star})^{n_d} \ar[r] \& 0 \end{tikzcd}.
\end{align*}
Since all maps (except for the first one) are induced by radical endomorphisms of $E_m$, we observe that the images lie in \( \operatorname{Rad} \End(E_m) \cdot \Hom_{\mathscr{C}}(E_m, \Sigma E_m^{\star}) \). Hence, by the Nakayma lemma, we can see iteratively from right to left that \( n_d = \cdots = n_1 = 0 \). Now, since $\Hom_{\mathscr{C}}(E_m, \sus E_m^{\star})^n$ is the epimorphic image of \( \End_{\mathscr{C}}(\sus E_m^{\star}) \) as a module over that endomorphism ring, and this endomorphism ring is basic by assumption, we have $n \leq 1$. Together with the inequality from the start of the proof this means \( T_0 = E_m \).

Finally, applying $\Hom_{\mathscr{C}}(\overline{T}, -)$ to the \( (d+2) \)-angle above, we obtain the exact sequence
\[ \begin{tikzcd}[sep=small,cramped] \Hom_{\mathscr{C}}(\overline{T}, T_1) \ar[r] & \Hom_{\mathscr{C}}(\overline{T}, T_0) \ar[r] & 0 \end{tikzcd} .\]
So the map $T_1 \to T_0 = E_m$ is a right $\overline{T}$-approximation.

\medskip
\paragraph{(3) \( \Longrightarrow \) (2)}

Applying $\Hom_{\mathscr{C}}(\overline{T}, -)$ to the right exchange $(d+2)$-angle, we obtain an exact sequence
\[ \begin{tikzcd}[sep=small,cramped] \Hom_{\mathscr{C}}(\overline{T}, \overline{T}_1) \ar[r,"f_*"] & \Hom_{\mathscr{C}}(\overline{T}, E_m) \ar[r] & \Hom_{\mathscr{C}}(\overline{T}, \sus E_m^{\star}) \ar[r] & \Hom_{\mathscr{C}}(\overline{T}, \sus \overline{T}_d) \end{tikzcd}. \]
Since $f$ is a right $\overline{T}$ approximation the map $f_*$ is surjective. Since $T$ is cluster tilting the rightmost term above vanishes. So we see that $\Hom_{\mathscr{C}}(\overline{T}, \sus E_m^{\star}) = 0$.

By the 2-Calabi-Yau property this also means that
\( \Hom_{\mathscr{C}}(E_m^{\star}, \sus \overline{T})  = \D \Hom_{\mathscr{C}}(\sus \overline{T}, \sus^2 E_m^{\star}) = 0 \).

Using the right exchange $(d+2)$-angle again, in each argument, we obtain the following commutative diagram, where the three term row and the three term column are exact.
\[ \begin{tikzcd}
\Hom_{\mathscr{C}}(\overline{T}_d, E_m) \ar[r] \ar[d] & |[draw=black,dashed]| \Hom_{\mathscr{C}}(\overline{T}_d, \sus E_m^{\star}) \ar[d] \\
\Hom_{\mathscr{C}}(E_m^{\star}, E_m) \ar[r] \ar[d] & \Hom_{\mathscr{C}}(E_m^{\star}, \sus E_m^{\star}) \ar[r] & |[draw=black,dashed]| \Hom_{\mathscr{C}}(E_m^{\star}, \sus \overline{T}_d) \\
|[draw=black,dashed]| \Hom_{\mathscr{C}}(\sus^{-1} E_m, E_m)
\end{tikzcd} \]
Note that the spaces in the dashed boxes are zero by what we have seen above and the assumption that $T$ is cluster tilting.

It follows that the composition \( \Hom_{\mathscr{C}}(\overline{T}_d, E_m) \to \Hom_{\mathscr{C}}(E_m^{\star}, \sus E_m^{\star}) \) is both surjective and zero, and therefore \( \Hom_{\mathscr{C}}(E_m^{\star}, \sus E_m^{\star}) = 0 \).

\medskip
\paragraph{(2) \( \Longleftrightarrow \) (4)} This equivalence is dual to the equivalence of (2) and (3) which we have proven above.

\medskip
\paragraph{(3) \( \Longrightarrow \) (5)}

Let \( K \in \mathscr{C} \), and let
\[ \begin{tikzcd}[sep=small,cramped] \sus K \ar[r] & T_d \ar[r] & \cdots \ar[r] & T_0 \ar[r] & K \end{tikzcd} \]
be a resolving \( (d+2) \)-angle with \( T_i \in \add T \) and such that all the maps \( T_i \to T_{i-1} \) are radical morphisms. %(Such a resolving \( (d+2) \)-angle can be obtained from an arbitrary one by splitting off all isomorphisms between the \(T_i\).)
 
Assume that $E_m$ is a direct summand of $T_i$ for some $i \in \{1, \ldots, d-1\}$. Then, rotating the right exchange $(d+2)$-angle, we have two $(d+2)$-angles as in the following diagram.
\[ \begin{tikzcd}
\cdots \ar[r] & T_{i+1} \ar[r] \ar[d,dashed] & T_i \ar[r] \ar[d,"\substack{\text{split}\\ \text{epi}}"'] & T_{i-1} \ar[r] \ar[d,dashed] \ar[ld,dotted] & \cdots \\
\cdots \ar[r] & \overline{T}_1 \ar[r] & E_m \ar[r] & \sus E_m^{\star} \ar[r] & \sus \overline{T}_d \ar[r] & \cdots
\end{tikzcd} \]
The map $\overline{T}_1 \to E_m$ is a right $\overline{T}$-approximation of $E_m$. Since $E_m$ does not have any loops attached to it, this map is also a radical $T$-approximation of $E_m$. Thus the composition $T_{i+1} \to T_i \to E_m$ factors through $\overline{T}_1$ as indicated above by the left dashed arrow. We then obtain a morphism of $(d+2)$-angles, and in particular the right dashed arrow.

Now the map $E_m \to \sus E_m^{\star}$ is a right $T$-approximation (since \( \Hom_{\mathscr{C}}(T, \Sigma \overline{T}_d) = 0 \)), hence the right dashed map factors through it as indicated by the dotted arrow above. Hence we have that
\begin{align*} \left[ E_m \to \sus E_m^{\star} \right] & = \left[ E_m \to \sus E_m^{\star} \right] \circ \left[ T_i \to \sus E_m \right] \circ \overbrace{\left[ E_m \to T_i \right]}^{\text{split mono}} \\
& = \left[ T_{i-1} \to \sus E_m^{\star} \right] \circ \left[ T_i \to T_{i-1} \right] \circ \left[ E_m \to T_i \right] \\
& = \left[ E_m \to \sus E_m^{\star} \right] \circ \underbrace{\left[ T_{i-1} \to E_m \right] \circ \left[ T_i \to T_{i-1} \right] \circ \left[ E_m \to T_i \right]}_{\in \operatorname{Rad} \End_{\mathscr{C}}(E_m)}
\end{align*}
and so the map \( E_m \to \sus E_m^{\star} \) vanishes. This however means that the map \( \overline{T}_1 \to E_m \) is a split epimorphism, which cannot happen because \(E_m \not\in \add \overline{T} \).

\medskip
\paragraph{(1) \(\Longrightarrow\) (2)} Immediate.

\medskip
\paragraph{(2) \( \Longrightarrow \) (1)}

By the implications we already know, we may assume additionally that (3), (4) and (5) hold. That is, we have a right and a left exchange triangle connecting \(E_m\) and \(E_m^{\star}\) and we know that any \( K \) has a resolving \((d+2)\)-angle without \( E_m \) appearing in the middle terms.

Since rigidity is assumed in (2), we only need to show that $\overline{T} \oplus E_m^{\star}$ has the resolving property. So let $K \in \mathscr{C}$. Then there is a $(d+2)$-angle
\[ \begin{tikzcd}[sep=small,cramped] \sus^{-1} K \ar[r] & T_d \ar[r] & T_{d-1} \ar[r] & \cdots \ar[r] & T_1 \ar[r] & T_0 \ar[r] & K \end{tikzcd} \]
with $T_i \in \add T$, and we may assume by (5) that \( T_i \in \add \overline{T} \) for \( i \in \{1, \ldots, d-1 \} \). We write \( T_0 = \tilde{T}_0 \oplus (E_m)^a \) and \( T_d = \tilde{T}_d \oplus (E_m)^b \), with \( \tilde{T}_0 \) and \( \tilde{T}_d \in \add \overline{T} \). First we consider the following diagram, where the top row is a sum of copies of the right exchange \((d+2)\)-angle.
\[ \begin{tikzcd}[column sep=small]
(E_m^{\star})^a \ar[r] & (\overline{T}_d)^a \ar[r] & (\overline{T}_{d-1})^a \ar[r] & \cdots \ar[r] & (\overline{T}_1)^a \ar[r,"(f)^a"] & (E_m)^a \ar[r] & (\sus E_m^{\star})^a \\
\sus^{-1} K \ar[r] & T_d \ar[r] & T_{d-1} \ar[r] & \cdots \ar[r] & T_1 \ar[r] \ar[u,dashed] & \tilde{T}_0 \oplus (E_m)^a \ar[r] \ar[u,"(0 \; 1)"] & K
\end{tikzcd} \]
Since $f$ is a right $\overline{T}$-approximation, and \(T_1 \in \add \overline{T} \) we get the factorization indicated by the dashed arrow.

We may complete these vertical morphisms to a good morphism of $(d+2)$-angles, and consider the cone $(d+2)$-angle. After splitting off the identity on \( (E_m)^a \) this looks like the upper row in the following diagram.
\[ \tikz{ \node [scale=.9]{
\begin{tikzcd}[column sep=tiny,cramped]
\sus^{-1} K \ar[r] & \tilde{T}_d \oplus (E_m)^b \oplus (E_m^{\star})^a \ar[r] & T_{d-1} \oplus (\overline{T}_d)^a \ar[r] & \cdots \ar[r] & T_1 \oplus (\overline{T}_2)^a \ar[r] & \tilde{T}_0 \oplus (\overline{T}_1)^a \ar[r] & K \\
(\sus^{-1} E_m^{\star})^b \ar[r] & (E_m)^b \ar[r] \ar[u,"(0 \; 1 \; 0)"] & (\overline{T}^1)^b \ar[r] \ar[u,dashed] & \cdots \ar[r] & (\overline{T}^{d-1})^b \ar[r] & (\overline{T}^d)^b \ar[r] & (E_m^{\star})^b 
\end{tikzcd}
}; } \]
The lower row consists of \( b \) copies of the left exchange \( (d+2) \)-angle. Dual to the first step, we observe that since \( g \) is a left \( \overline{T} \)-approximation and \( T_{d-1} \oplus (\overline{T}_d)^a \in \add \overline{T} \) we can find the dashed arrow, and hence a good morphism of \( (d+2) \)-angles. Again we take the cone and split off the identity on \( (E_m)^b \), obtaining a \( (d+2) \)-angle resolving \( K \) by terms in \( \add \overline{T} \oplus E_m^{\star} \).
\end{proof}

\begin{proof}[{Proof of Theorem~\ref{thm.mutability_for_clustercat}}]
We first observe that, by Proposition~\ref{prop.C_no_loops}, \( E_m \) has no loops in the quiver of \( T \). Therefore Theorem~\ref{thm.find_replacement} applies. Thus it suffices to show that (5) implies (3).

Let \( f \) be a minimal right \( \overline{T} \)-approximation of \(E_m\), and complete it to a \( (d+2) \)-angle
\[ \begin{tikzcd}[sep=small] \sus^{-1} E_m \ar[r] & H_{d+1} \ar[r] & H_d \ar[r] & \cdots \ar[r] & H_2 \ar[r] & \overline{T}_1 \ar[r,"f"] & E_m, \end{tikzcd} \]
in which all maps lie in the radical of $\mathscr{C}$.

First assume that $H_i \not\in \add T$ for some $i \in \{2, \ldots, d\}$. Let $i$ be minimal with this property, and let $K$ be an indecomposable direct summand of $H_i$ which does not lie in $\add T$. Since $T$ is cluster tilting there is a $(d+2)$-angle
\[ \begin{tikzcd}[sep=small,cramped] K \ar[r] & T_d \ar[r] & \cdots \ar[r] & T_1 \ar[r] & T_0 \ar[r] & \sus K \end{tikzcd} \]
with $T_i \in \add T$, and by (5) we may choose it such that $E_m$ is not a direct summand of $T_j$ for $j \in \{1, \ldots, d-1\}$. We obtain the following diagram, where the rows are (parts of) rotations of these $(d+2)$-angles.
\[ \begin{tikzcd}
H_{i+1} \ar[r] & H_i \ar[r] & H_{i-1} \ar[r] & \cdots \ar[r] & H_2 \ar[r] & \overline{T}_1 \ar[r,"f"] & E_m \\
\sus^{-1} T_0 \ar[r] \ar[u,dashed] & K \ar[r] \ar[u] \ar[lu,dotted] & T_d \ar[r] \ar[u] \ar[lu,dotted] & \cdots \ar[r] \ar[lu,dotted] & T_{d-i+3} \ar[r] \ar[u,dashed] \ar[lu,dotted] & T_{d-i+2} \ar[r] \ar[u,dashed] \ar[lu,dotted] & T_{d-i+1} \ar[u,dashed] \ar[lu,dotted]
\end{tikzcd} \]
Here, the left solid vertical map is a split monomorphism (which exists by assumption). The right vertical solid map exists since the map $K \to T_d$ is a left $T$-approximation, and $H_{i-1} \in \add T$ by the minimality of $i$. (In the extreme case that \( i = 2 \) we use \( \overline{T}_1 \) instead of \( H_{i-1} \), but the argument remains the same.) Then we get a morphism of $(d+2)$-angles as indicated by the dashed vertical maps. Note that since $i \in \{2, \ldots, d\}$ we have $d-i+1 \in \{1, \ldots, d-1\}$, so $T_{d-i+1} \in \add \overline{T}$. Hence, since $f$ is a right $\overline{T}$-approximation, the rightmost dashed vertical map factors through $f$, as indicated by the rightmost dotted arrow. Now we propagate this homotopy to the left, obtaining the other dotted arrows. Looking at the position of $K$ we see that the split monomorphism $K \to H_i$ can now be written as the sum of the compositions $K \to H_{i+1} \to H_i$ and $K \to T_d \to H_i$, both of which lie in the radical of $\mathscr{C}$. This is a contradiction. Hence we have shown that for $i \in \{2, \ldots, d\}$ we have $H_i \in \add T$.

Finally note that by Proposition~\ref{prop.C_no_loops} we may choose the \( (d+2) \)-angle such that no other term has summands \( E_m \). It follows that \( H_2 \) up to \( H_d \) are in \( \add \overline{T} \), and thus that we have found a right mutation \( (d+2) \)-angle.
\end{proof}

\section{The index} \label{sect.index}

\begin{definition}
For an additive category \( \mathscr{A} \) we denote by \( \Gr(\mathscr{A}) \) the split Grothendieck group of \( \mathscr{A} \), that is the free abelian group on the objects of \( \mathscr{A} \) subject to the identifications \( [A_1 \oplus A_2] = [A_1] + [A_2] \). If the Krull-Schmidt theorem holds for \( \mathscr{A} \) then this is just the free abelian group on the indecomposables in \( \mathscr{A} \).

By abuse of notation, we write \( \Gr (A) \) instead of \( \Gr(\add A) \) for an object \(A \in \mathscr{A}\).
\end{definition}

\begin{definition}
Let \( T \) be a cluster tilting object in a \( 2 \)-Calabi-Yau \( (d+2) \)-angulated category \( \mathscr{C} \). Then, for any object \( K \in \mathscr{C} \) the \emph{index} of \( K \) with respect to \( T \) is
\[ \index_{T}(K) = \sum_{i=0}^d (-1)^i [T_i] \in \Gr(T), \]
provided that the resolution of \( K \) in \( \add T \) is
\[ \Sigma^{-1} K \to T_d \to \cdots \to T_0 \to K. \]
\end{definition}

\begin{remark}
Note that splitting off any isomorphisms between summands of two adjacent \( T_i \) does not affect the index, since this summand will appear twice with opposite signs. Since any resolution can be turned into one only having radical morphisms between the \( T_i \) by splitting off summands, and since a resolution only containing radical morphisms is unique by (a version of) Lemma~\ref{lem.d-cone_unique}, we see that the index is well-defined. 
\end{remark} 

This definition of index was given by Jørgensen \cite{J1}, and also plays an important role in \cite{R1, R2, R3, JS1}.

Suppose that the cluster tilting object $T^*$ is obtained from
  the cluster tilting object $T$ by a single mutation. Our goal in this section is to explain the relationship between $\index_T(K)$ and $\index_{T^*}(K)$.

The paper \cite{J1} is also concerned with relations among different indices, but with the important distinction that Jørgensen is relating indices of different modules with respect to the same cluster tilting object, whereas we are relating indices of the same object with respect to different cluster tilting objects. These two approaches produce quite different dynamics, as is already visible in the cluster algebraic setting ($d=1$).

Our ultimate goal for this section is to show that (for odd \( d \)) the index satisfies higher tropical coefficient dynamics. Let us start by making precise what that means.

\begin{definition} \label{def.higher_trop_comb}
Let $d$ be an odd positive integer, and let \( \mathscr{C} \) be a \( 2 \)-Calabi-Yau \( (d+2) \)-angulated category. Assume all cluster tilting objects in \( \mathscr{C} \) have exactly \( n \) summands.
 
A function
\[ f \colon \{ \text{(ordered) cluster tilting objects in \( \mathscr{C} \)} \} \mapsto  \mathbb{Z}^n \]
satisfies \emph{higher tropical coefficient dynamics} if the following mutation formula holds:

For \( T \) and \( T^{\star} \) two cluster tilting objects related by a single mutation in \( E_m \) as in Definition~\ref{def.mutation}, using the notation of Theorem~\ref{thm.find_replacement} we have
\[ f(T^{\star})_i = \begin{cases}
- f(T)_i & \text{ if } i = m \\
f(T)_i - \sum_{j=1}^d (-1)^j \mult{E_i}{\overline{T}_j} f(T)_m & \text{ if } f(T)_m \geq 0 \\
f(T)_i - \sum_{j=1}^d (-1)^j \mult{E_i}{\overline{T}^j} f(T)_m & \text{ if } f(T)_m \leq 0,
\end{cases} \]
generalizing the formula for \( \mu^T_m \) of Section~\ref{2-2}.
\end{definition}

\begin{remark} \label{rem.higher_dynamics}
Note that the alternating sums in the mutation formula are very close to the ones appearing in the definition of the index of \( \Sigma E_m^{\star} \) and \( E_m^{\star} \), respectively. Thus the formula may also be written as
\[ f(T^{\star})_i = \begin{cases}
- f(T)_i & \text{ if } i = m \\
f(T)_i - \mult{E_i}{\index_T(\Sigma E_m^{\star})} f(T)_m & \text{ if } f(T)_m \geq 0 \\
f(T)_i + \mult{E_i}{\index_T(E_m^{\star})} f(T)_m & \text{ if } f(T)_m \leq 0,
\end{cases} \]
\end{remark}

We start our investigation of the effect of mutation on indices with the following proposition, which summarizes the information we obtain relatively directly from Theorem~\ref{thm.find_replacement}.

\begin{proposition} \label{prop.g_mut_I}
Let $T$ be a cluster tilting object in $\mathscr{C}$, and let $E_m$ be a mutable summand of $T = \overline{T} \oplus E_m$ without loops attached to it. Let $T^{\star} = \overline{T} \oplus E_m^{\star}$ be the mutated cluster tilting object.

Let $K \in \mathscr{C}$ and let
\[ \begin{tikzcd}[sep=small,cramped] \sus^{-1} K \ar[r] & T_d \ar[r] & \cdots \ar[r] & T_0 \ar[r] & K \end{tikzcd} \]
be a $(d+2)$-angle with $T_i \in \add T$, and such that the maps between the $T_i$ lie in the radical of $\mathscr{C}$.

Then
\begin{align*}
\index_{T^{\star}}(K) = \index_T(K) & + \mult{E_m}{T_0} \cdot \left( (-1)^d [E_m^{\star}] - \index_T(\Sigma E_m^{\star}) \right) \\
& + \mult{E_m}{T_d} \cdot \left( [E_m^{\star}] - \index_T(E_m^{\star}) \right)
\end{align*}
(In particular this includes the claim that the term on the right hand side, which a priori is an element of \( \Gr (T^{\star} \oplus E_m) \), in fact lies in the subgroup \( \Gr(T^{\star}) \).)
\end{proposition}

\begin{proof}
This follows from the construction of a \( (d+2) \)-angle resolving \( K \) in the proof of Theorem~\ref{thm.find_replacement}. In that proof we wrote \( a = \mult{E_m}{T_0} \) and \( b = \mult{E_m}{T_d} \), and constructed a \( (d+2) \)-angle showing that (in the notation of that proof)
\begin{align*}
\index_{T^{\star}} (K) & = [\tilde{T}_0] + a [\overline{T}_1] + b [E_m^{\star}] + \sum_{i=1}^{d-1} (-1)^i \left( [T_i] + a [\overline{T}_{i+1}] + b [\overline{T}^{d+1-i}] \right) \\
& \qquad + (-1)^d \left( \tilde{T}_d + a [ E_m^{\star} ] + b [ \overline{T}^1 ] \right) \\
& = \index_T(K) - a [E_m] - (-1)^d b [E_m] \\
& \qquad - a \left( \index_T(\Sigma E_m^{\star}) - [E_m] - (-1)^d [E_m^{\star}] \right) \\
& \qquad - b \left( \index_T(E_m^{\star}) - (-1)^d [E_m] - [E_m^{\star}] \right) \\
& = \index_T(K) - a \left( \index_T(\Sigma E_m^{\star}) - (-1)^d [E_m^{\star}] \right) \\
& \qquad - b \left( \index_T(E_m^{\star}) - [E_m^{\star}] \right)
\end{align*}
as claimed.
\end{proof}

\begin{lemma} \label{lemma.rigid_only_one_end}
Let $K \in \mathscr{C}$ be rigid (i.e.\ \( \Hom_{\mathscr{C}}(K, \sus K) = 0 \)), and let
\[ \begin{tikzcd}[sep=small,cramped] \sus^{-1} K \ar[r] & T_d \ar[r] & T_{d-1} \ar[r] & \cdots \ar[r] & T_1 \ar[r] & T_0 \ar[r] & K \end{tikzcd} \]
be a resolving $(d+2)$-angle, such that $T_i \in \add T$, and all the maps $T_i \to T_{i-1}$ lie in the radical of $\mathscr{C}$.

Then $T_d$ and $T_0$ have no common direct summands.
\end{lemma}

\begin{proof}
Assume $T_d$ and $T_0$ have a common direct summand. Then we have a map $f \colon T_d \to T_0$ which does not lie in the radical of $\mathscr{C}$. We consider the following diagram.
\[ \begin{tikzcd}
\sus^{-1} K \ar[r] & T_d \ar[r] \ar[d,"f"'] \ar[ld,dotted] & T_{d-1} \ar[d,dashed] \ar[ld, dotted] \\
T_1 \ar[r] & T_0 \ar[r] & K \ar[r] & \sus T_d 
\end{tikzcd} \]
Since the composition from $\sus^{-1} K$ to $K$ is $0$, we obtain the dashed arrow making the square commutative. Since $\Hom_{\mathscr{C}}(T_{d-1}, \sus T_d) = 0$ the dashed arrow factors through $T_0 \to K$ as indicated by the right dotted arrow. Now the compositions $T_d \overset{f}{\to} T_0 \to K$ and $T_d \to T_{d-1} \to T_0 \to K$ coincide, so the difference \( f - [T_d \to T_{d-1} \to T_0] \) factors through \( [T_1 \to T_0] \), as indicated by the left dotted arrow. However, both \( [T_d \to T_{d-1} ] \) and \( [T_1 \to T_0] \) are radical morphisms, contradicting the fact that \( f \) is not a radical morphism.
\end{proof}

\begin{corollary}\label{cor.main_thm_alld}
Let \(T\) be a cluster tilting object in \( \mathscr{C} \), and let $E_m$ be a mutable summand of $T = \overline{T} \oplus E_m$ without loops attached to it. Let $T^{\star} = \overline{T} \oplus E_m^{\star}$ be the mutated cluster tilting object.

Let \( K \in \mathscr{C} \), and let
\[ \begin{tikzcd}[sep=small,cramped] \sus^{-1} K \ar[r] & T_d \ar[r] & \cdots \ar[r] & T_0 \ar[r] & K \end{tikzcd} \]
be a resolving $(d+2)$-angle, with $T_i \in \add T$, and such that the maps between the $T_i$ lie in the radical of $\mathscr{C}$.
Then
\begin{align*}
\index_{T^{\star}}(K) & = \index_T(K) + (-1)^d \mult{E_m}{\index_T(K)} [E_m^{\star}] \\
& \qquad - | \mult{E_m}{\index_T(K)} | \cdot \begin{cases} \index_T(\Sigma E_m^{\star}) & \text{if } \mult{E_m}{T_d} = 0 \\  \index_T(E_m^{\star}) & \text{if } \mult{E_m}{T_0} = 0 \end{cases}
\end{align*}

In particular, if $K \in \mathscr{C}$ is rigid (\(\Hom_{\mathscr{C}}(K, \sus K) = 0\)), then by Lemma~\ref{lemma.rigid_only_one_end} (at least) one of the cases applies.
\end{corollary}

\begin{proof}
This immediately follows from Proposition~\ref{prop.g_mut_I}. Note that if \(\mult{E_m}{T_d} = 0 \) then \( \mult{E_m}{T_0} = \mult{E_m}{\index_T(K)} \) and similarly in the other case.
\end{proof}

For odd $d$ we obtain the following result as an immediate consequence of the previous corollary.

\begin{corollary} \label{cor.main_thm}
Assume $d$ is odd. Let $T$ be a cluster tilting object in $\mathscr{C}$, and let $E_m$ be a mutable summand of $T = \overline{T} \oplus E_m$ without loops attached to it. Let $T^{\star} = \overline{T} \oplus E_m^{\star}$ be the mutated cluster tilting object.

Let $K \in \mathscr{C}$ be rigid (i.e.\ $\Hom_{\mathscr{C}}(K, \sus K) = 0$).

Then
\begin{align*}
\index_{T^{\star}}(K) & = \index_T(K) - \mult{E_m}{\index_T(K)} [E_m^{\star}] \\
& \qquad \begin{cases} - \mult{E_m}{\index_T(K)} \index_T(\Sigma E_m^{\star}) & \text{if } \mult{E_m}{\index_T(K)} \geq 0 \\  + \mult{E_m}{\index_T(K)} \index_T(E_m^{\star}) & \text{if } \mult{E_m}{\index_T(K)} \leq 0 \end{cases}
\end{align*}
\end{corollary}

Comparing to Remark~\ref{rem.higher_dynamics}, we see that Corollary~\ref{cor.main_thm} precisely says that the index obeys higher tropical coefficient dynamics.

The important advantage of Corollary \ref{cor.main_thm} over
Corollary \ref {cor.main_thm_alld} is that in Corollary~\ref{cor.main_thm}, in order to know which case of the formula to apply, we ask whether $[\index_T(K):E_m]$ is
non-negative or non-positive, which, obviously, depends only on
$\index_T(K)$. In contrast, when applying the Corollary \ref {cor.main_thm_alld}, we must know whether
$[T_d:E_m]$ is non-negative
or $[T_0:E_m]$ is non-negative, and when $d$ is even, the index will not tell us this, since either term would contribute positively to the index. 

In other words, Corollary \ref{cor.main_thm} tells us how to determine
$\index_{T^*}(K)$ knowing nothing about $K$ other than $\index_T(K)$.

\section{Higher shear coordinates}\label{laminations}

We now specialize to the situation studied in \cite{OT}. $A^d_n$ is the
$(d-1)$ higher Auslander algebra of linearly oriented $A_n$, and 
$\mathscr C^d_n$ is its $(d+2)$-angulated cluster category.

\subsection{Reminders about $\mathscr C^d_n$}\label{recall}

The following description of \( \mathscr{C}^d_n \) can be found in \cite{OT}, in particular in Section~6.
Define $\IndexC{n}$ to be the set of $d+1$-sets from
$\{0,1,\dots,n+2d\}$ with no two elements cyclically consecutive (i.e., if
$i\in I$, then $i+1\not\in I$, with addition taken modulo $n+2d+1$).

The indecomposable objects of $\mathscr C^d_n$ are indexed by $\IndexC{n}$.
For $(i_0,\dots,i_d)\in \IndexC{n}$, we write $O_{(i_0,\dots,i_d)}$ for the
corresponding indecomposable object of $\mathscr C^d_n$.

We say that \( I \) and \( J \) in \( \IndexC{n} \) \emph{intertwine} if the elements of $I$ and $J$ alternate cyclically; that is to say,
it is possible to list the elements of $I$ as $i_0<\dots<i_d$ and the elements of $J$ as $j_0<\dots<j_d$ such that 
either $i_0<j_0<i_1<j_1\dots<i_d<j_d$
or $j_0<i_0<j_1<i_1<\dots<j_d<i_d$. In this case we say that $I$ and $J$
intertwine. We say that two subsets $X$ and $Y$ of
$\IndexC{n}$ do not intertwine if no $I\in X$ and $J\in Y$ intertwine.

For $I,J\in \IndexC{n}$, we have that $\Hom(O_I,O_J[d])$ is one-dimensional if $I$ and $J$ intertwine, and zero-dimensional otherwise.

It follows that the indices of the summands of a cluster tilting object are necessarily non-intertwining. In fact, such collection of indices are exactly the non-intertwining sets of maximal size.

The action of the suspension has a nice interpretation in terms of this indexing. For $I\in\IndexC{n}$, 
$\Sigma O_{(i_0,\dots,i_d)}\simeq O_{(i_0-1,\dots,i_d-1)}$, unless $i_0=1$, in which case,
$\Sigma O_{(1,\dots,i_d)}\simeq O_{(i_1-1,\dots,i_d-1,n+2d+1)}$. 

The set $\IndexC{n}$ also has an interpretation in terms of convex geometry.
Choose real numbers
$v_0<\dots<v_{n+2d}$.
Consider the moment curve, defined by $p(t)=(t,t^2,\dots,t^{2d})$. Take the convex hull of the points $p(v_i)$; this is a cyclic polytope $P$ with vertices $p(v_i)$.
An internal $d$-simplex of $P$ is a simplex whose vertices are vertices of
$P$ and which does not lie entirely within the boundary of $P$. The internal
$d$-simplices of $P$ are the convex hulls of $p(i_0), \dots, p(i_d)$ for
$(i_0,\dots,i_d)\in \IndexC{n}$.

The cluster tilting objects in $\mathscr C^n_d$ are in bijective correspondence with triangulations of $P$ (i.e., subdivisions of $P$ into simplices, all of
whose vertices are vertices of $P$). Under this correspondence, a triangulation corresponds to the
direct sum of the indecomposable objects corresponding to the internal $d$-simplices appearing in the triangulation.

For $I\in \IndexC{n}$, if  $O_I$ is a mutable summand of $T$,
which mutates to $O_J$ for $J=(j_0,\dots,j_d)$, then the triangulation
corresponding to $T$ restricts to a triangulation of the convex hull of the
points $p(i_0),\dots,p(i_d),p(j_0),\dots,p(j_d)$. There are two triangulations
of a $2d$-dimensional cyclic polytope with $2d+2$ vertices, one with internal $d$-simplex $I$, the other with internal $d$-simplex $J$; on the level of the triangulation, mutating $T$ at $O_I$ replaces the one by the other.

\subsection{Shear coordinates for $d=1$}\label{six-two}
We begin in the
$d=1$ setting, where $A^1_n$ is simply the path algebra of linearly oriented
$A_n$ quiver, and $\mathscr C^1_n$ is the corresponding
cluster category $D^b(\mod A^1_n)/\Sigma^{-2}S$.
The combinatorial model for all $d$ described in the previous subsection restricts to something simpler here. The convex hull of $n+3$ points on the moment curve is a convex $(n+3)$-gon, and the choice to locate the vertices on the moment curve does not make any difference: we can simply think of any convex polygon in the plane, with vertices numbered counterclockwise from 0 to $n+2$.

The indecomposable objects of $\mathscr C^1_n$ are of the form
$O_{\{i_0,i_1\}}$, where the possible pairs $\{i_0,i_1\}$ correspond to the diagonals of the polygon.

For $K$ an indecomposable object of 
$\mathscr C^1_n$, we write $\gamma(K)$  for the corresponding diagonal of the polygon.

Let $T=\bigoplus_{i=1}^{n} E_i$ be a cluster tilting object. The diagonals corresponding to summands of $T$
define a triangulation.

We will be interested in what we refer to as ``arcs'' in the polygon.
These are curves whose endpoints lie on two different boundary segments
of the polygon (and, in particular, not at vertices). 
The shear coordinates of an arc $\alpha$ with respect to a triangulation $T$ are given
as follows. Each diagonal $\gamma(E_i)$ defines a quadrilateral composed of the two
triangles with edges corresponding to summands of $T$ or edges of the polygon,
and having $\gamma(E_i)$ as an edge. In order for $\sh_T(\alpha)_i$ to
be non-zero, $\alpha$ must enter and leave the quadrilateral surrounding
$\gamma(E_i)$ on opposite
sides. In this case, we have $\sh_T(\alpha)_i=1$ or $-1$, according to the diagram below.
(We also include an example of an arc
  $\alpha$ with $\sh_T(\alpha)_i=0$.)

\[ \begin{tabular}{c@{\hspace{1.5cm}}c@{\hspace{1.5cm}}c}
\begin{tikzpicture}
 \coordinate (A) at (0.3,0.3);
 \coordinate (B) at (1.3,1.1);
 \coordinate (C) at (2.3,-.5);
 \coordinate (D) at (1,-1);
 \draw (A) -- (B);
 \draw (B) -- (C);
 \draw (C) -- (D);
 \draw (D) -- (A);
 \draw (A) -- node [above=-1pt] {\(E_i\)} (C);
 \draw [dashed,bend left=10] ($(A)!0.67!(D)$) to node [below=-1pt] {\(\alpha\)} ($(B)!0.75!(C)$);
\end{tikzpicture}
&
\begin{tikzpicture}
 \coordinate (A) at (0.3,0.3);
 \coordinate (B) at (1.3,1.1);
 \coordinate (C) at (2.3,-.5);
 \coordinate (D) at (1,-1);
 \draw (A) -- (B);
 \draw (B) -- (C);
 \draw (C) -- (D);
 \draw (D) -- (A);
 \draw (A) -- node [above=-1pt] {\(E_i\)} (C);
 \draw [dashed,bend left=15] ($(A)!0.3!(B)$) to node [near end,left=-1pt] {\(\alpha\)} ($(C)!0.75!(D)$);
\end{tikzpicture}
&
\begin{tikzpicture}
 \coordinate (A) at (0.3,0.3);
 \coordinate (B) at (1.3,1.1);
 \coordinate (C) at (2.3,-.5);
 \coordinate (D) at (1,-1);
 \draw (A) -- (B);
 \draw (B) -- (C);
 \draw (C) -- (D);
 \draw (D) -- (A);
 \draw (A) -- node [above=-1pt] {\(E_i\)} (C);
 \draw [dashed,bend right] ($(A)!0.67!(D)$) to node [near start,right=-1pt] {\(\alpha\)} ($(A)!0.4!(B)$);
\end{tikzpicture} \\
\( \sh_T(\alpha)_i=1 \) & \( \sh_T(\alpha)_i=-1 \) & \( \sh_T(\alpha)_i=0 \)
\end{tabular} \]

Note that our choice of signs in shear coordinates is the opposite of the usual one. This is a result of the choice we have made for the orientation of the numbering of the vertices of the polygon, which is consistent with the interpretation in terms of points on the moment curve.

For an arc $\alpha$, we define $\alpha^-$ to be the line segment that results by moving its endpoints counterclockwise until they reach vertices of the polygon.

The following theorem follows immediately from putting together two facts that are well-known, though possibly to different sets of people.

\begin{theorem} Let $T$ be a cluster tilting object, and let $\alpha$ be an arc. If $\alpha^-$ is not a diagonal of the polygon, then $\sh_T(\alpha)=0$. Otherwise, 
  $\sh_T(\alpha)=\index_T(O_{\alpha^-})$.
  \end{theorem}

\begin{proof} If $\alpha^-$ is not a diagonal of the polygon, then
  $\alpha^-$ is an edge, and the endpoints of $\alpha$ are on adjacent edges of the polygon. In this case, it is clear that the shear coordinates of $\alpha$ are zero, as it can never enter and leave a quadrilateral by opposite edges.

  Otherwise, $\sh_T(\alpha)$ agrees with the $g$-vector of the cluster variable corresponding to $O_{\alpha^-}$ with respect to the initial variables corresponding to the summands of $T$ by Fomin--Thurston \cite[Proposition 17.3]{FT}.

  The result of Dehy--Keller \cite{DK} mentioned earlier implies, as they explain, that the index of an indecomposable object $K$ in $\mathcal C^1_n$ with respect to a cluster tilting object $T$ agrees with the $g$-vector of the cluster variable corresponding to $K$ with respect to the initial variables corresponding to the summands of $T$.
  
Putting these two facts together, the theorem is proved.\end{proof}

This result can be extended to collections of arcs not intersecting in their interiors. These are called \emph{laminations}. The corresponding objects are no longer indecomposable, but they are rigid.

For generalization to higher dimensions, it will be convenient
  to express the combinatorial essence of the geometric story we have been
  telling about arcs. Let $A=(a_0,a_1)\in\mathbb R^2$, where $a_0<a_1$, and neither
  $a_0$ nor $a_1$ belongs to $\{v_0,v_1,\dots,v_{n+2}\}$. We think of $A$ as corresponding
  to the arc $\alpha$ which is the intersection of the convex hull of
  $p(v_0),\dots,p(v_{n+2})$ with the line segment from $p(a_0)$ to $p(a_1)$.
  
  To describe $\alpha^-$, it is convenient to introduce some further notation.
  For a real number \( t \), let \( t^- \) denote the largest $v_i$ less than or equal to $t$. If $t<v_{0}$, define $t^-=v_{n+2d}$. Now define $A^-=\{a^-\mid a\in A\}$. Then $\alpha^-=A^-$.

For future use, let us similarly define $t^+$ to be the smallest $v_i$ greater than or equal to $t$ (or $v_0$ if there is none), and $A^+=\{a^+\mid a\in A\}$.

\subsection{Higher shear coordinates}
Now consider the case that $d>1$ is odd.

Pick $A$ a $(d+1)$-tuple of real numbers, disjoint from $\{v_i\}$. Then $A$ determines  a $(d+1)$-tuple of points on the moment curve, avoiding the vertices of $P$. We consider slicing $P$ by the $d$-dimensional plane which is the affine hull of $\{p(a)\mid a\in A\}$. We want to define shear coordinates for this slice with respect to $T$.

Let $E_m$ be a summand of $T$, and suppose that we want to compute the shear coordinate of $A$ corresponding to it. 
Suppose first that $E_m$ is mutable. In this case, there is a $(2d+2)$-tuple
$B=(b_0<b_1<\dots<b_{2d+1})$ such that $T$ restricts to a triangulation of
the convex hull of $p(b_0),\dots,p(b_{2d+1})$, and $E_m$ corresponds to one of the two
internal simplices of this polytope, whose vertices are the points corresponding to some $R\subseteq B$, where either $R=(b_0,b_2,\dots,b_{2d})$ or $R=(b_1,b_3,\dots,b_{2d+1})$. The points of $B$ divide the real line into $2d+2$ intervals
(with the interval below $b_0$ and the interval above $b_{2d+1}$ counting as the same interval). There are $d+1$ of these intervals which are immediately below some element of $R$, and there are $d+1$ which are immediately above some element of $R$. If there is an element of $A$ which lies in each of the intervals immediately below an element of $R$, the shear coordinate in position $E_m$ of $A$ is $-1$. If there is an element of $A$ which lies in each of the intervals immediately above an element of $R$, the shear coordinate is $1$. Otherwise, the shear coordinate is zero. 

Unfortunately, if $E_m$ is not mutable, then we do not have control over the configuration of the triangulation $T$ around the $d$-simplex corresponding to $E_m$, and it is not clear how to use the above approach to define shear coordinates.

We may however take inspiration from the fact that shear coordinates satisfy tropical coefficient dynamics, and define
\begin{definition}
A function \( \sh(A) \colon \{ \text{cluster tilting objects in } \mathscr{C} \} \to \mathbb{Z}^n \) defines higher shear coordinates of $A$ if 
\begin{enumerate}
\item \( \sh_T(A)_m \) is the shear coordinate of \( A \) with respect to \( E_m \) as defined above whenever \( E_m \) is a mutable summand of \( T \);
\item \( \sh(A) \) satisfies higher tropical coefficient dynamics (see Definition~\ref{def.higher_trop_comb}).
\end{enumerate}
\end{definition}

One easily sees that there is at most one function defining shear coordinates for a given \( A \): Start with any two that differ in some coordinate on a given \( T \). By the first property of shear coordinates they do agree on mutable summands. Then the second property implies that they still differ on the same coordinate of all mutations \( T^{\star} \) of \( T \). But under iterated mutation any summand will eventually become mutable, contradicting the first property we assumed.

We will show that a unique function defining shear coordinates does exist, by showing the following, where $A^-$ is defined as at the end of Section \ref{six-two}.

\begin{theorem}\label{shear-is-index}
There exists a (unique) function \( \sh(A) \) defining shear coordinates of \( A \). This function is given by
\[ sh_\bullet(A) = \index_\bullet(O_{A^-}) \]
where we set \( O_{A^-} = 0 \) whenever \( A^- \not\in \IndexC{n} \).
\end{theorem}

\begin{proof}
We have seen in Corollary \ref{cor.main_thm} that the index satisfies higher tropical coefficient dynamics. Thus it only remains to show that the index computes shear coordinates with respect to mutable summands as described above.

To that end, assume \( E_m \) is mutable, and let \( B \) and \( R \) be as above. For notational purposes we fix \( R = (b_0, b_2, \ldots, b_{2d}) \), and write \( R^* = (b_1, b_3, \ldots, b_{2d+1}) \), but nothing changes if the roles are reversed.

Note that
\[ \Hom(O_R, O_{A^-}) = \Hom(O_R, \Sigma O_{A^+}) = \begin{cases} k & \text{\( R \) and \( A^+ \) intertwine} \\ 0 & \text{otherwise.} \end{cases} \]

The left \( \overline{T} \)-approximation of \( E_m \) is given by the sum over all \( O_{R_i} \), where \( R_i \) is obtained from \( R \) be replacing \( b_{2i} \) by \( b_{2i+1} \) (see \cite[Theorem~6.3(3)]{OT}).

If \( A \) lies in the intervals directly above \( R \), then \( R \) and \( A^+ \) intertwine, but none of the \( R_i \) intertwine with \( A^+ \). It follows that there is a map from \( E_m \) to \( O_{A^-} \) which does not factor through a left \( \overline{T} \)-approximation of \( E_m \), and hence that \( E_m \) appears in the right \( T \)-approximation of \( O_{A^-} \).

By Theorem~\ref{thm.find_replacement}(5) and Lemma~\ref{lemma.rigid_only_one_end} this is the only contribution of \( E_m \) to the index of \( O_{A^-} \), so \( \mult{E_m}{\index_T(O_{A^-})} = 1 \) as desired.

\smallskip
If \( A \) lies in the intervals directly below \( R \), then it lies in the intervals directly above \( R^* \). The same argument as above shows that \( \mult{E_m^*}{\index_{T^*}(O_{A^-})} = 1 \), and it follows from our mutation formula that \( \mult{E_m}{\index_{T}(O_{A^-})} = -1 \).

\smallskip
Now assume that \( \mult{E_m}{\index_T(O_{A^-})} > 0 \). Then \( E_m \) appears as a summand in the right \( T \)-approximation of \( O_{A^-} \). In particular \( R \) and \( A^+ \) intertwine. Since the morphism \( E_m \to O_{A^-} \) does not factor through \( E_m \to \bigoplus_i O_{R_i} \) it necessarily induces a non-zero morphism \( \Sigma^{-1} E_m^* \to O_{A^-} \). In particular \( (b_1+1, b_3+1, \ldots, b_{2d+1}+1) \) and \( A^+ \) also intertwine. Thus, up to cyclic renumbering,
either \( b_{2i+1}+1 < a_i^+ < b_{2i+2} \), so \( A \) lies in the intervals immediately below \( R \), or \( b_{2i} < a_i^+ < b_{2i+1}+1 \), so \( A \) lies in the intervals immediately above \( R \). In the former case we have seen above that \( \mult{E_m}{\index_{T}(O_{A^-})} = -1 \), contradicting our current assumption. In the latter case we have seen that indeed \( \mult{E_m}{\index_T(O_{A^-})} = 1 \).

\smallskip
Finally, assume \( \mult{E_m}{\index_T(O_{A^-})} < 0 \). Then \( \mult{E_m^*}{\index_{T^*}(O_{A^-})} > 0 \), and thus \( A \) lies either in the intervals immediately below or the intervals immediately above \( R^* \). But this is the exact same condition as being in the intervals immediately above or immediately below \( R \).
\end{proof}

In fact, shear coordinates are also determined by the following different set of three properties, which may be easier to verify:
\begin{enumerate}
\item If $A^-$ does not define an internal $d$-simplex of $P$, then the shear coordinates of $A$ with respect to any $T$ are zero.
\item If $A^-$ does define an internal simplex of $P$, and the triangulation $T$ contains that simplex as a face, then the shear coordinates of $A$ with respect to $T$ is the unit vector $e_{A^-}$  
  (i.e., 1 in position $A^-$ and zero elsewhere).
\item If $T^*$ is obtained from $T$ by a mutation,
  the shear coordinates of $A$ with respect to $T^*$ are obtained from those with respect to $A$ by tropical coefficient mutation.
\end{enumerate}

For \( d=1 \) it is clear that these three properties uniquely describe shear
coordinates.

For higher \( d \), it is clear that there is at most one collection of shear coordinates satisfying the three properties above, but it is not clear that there are any: on the face of it, it isn't clear that tropical coefficient dynamics define a consistent collection of shear coordinates. However, one easily sees that the index of \( O_{A^-} \) does satisfy (1), by defining the corresponding object to be zero, (2), by definition of the index, and (3) by Corollary~\ref{cor.main_thm}.

As in the $d=1$ case, instead of considering $A$,
a single increasing $(d+1)$-tuple of real numbers, we could consider a collection $\mathcal A$ of increasing $(d+1)$-tuples all disjoint from $\{v_i\}$ and which are pairwise non-intertwining. Geometrically, this is equivalent to saying that the $d$-simplices corresponding to the elements of $\mathcal A$ do not
intersect.
The shear coordinates of such a collection of $(d+1)$-tuples 
can be defined as the sum of the shear coordinates of each of the $(d+1)$-tuples
individually.

We then have the following result.

\begin{corollary} Let $\mathcal A$ be a collection of increasing
  $(d+1)$-tuples of real
  numbers all disjoint from $\{v_i\}$, and which are pairwise non-intertwining.
Then $sh_\bullet(\mathcal A)$ satisfies higher tropical coefficient dynamics.
\end{corollary}

\begin{proof} We have that $$sh_\bullet(\mathcal A)=\index_\bullet(\bigoplus_{A\in\mathcal A} O_{A^-}),$$ and since $\bigoplus_{A\in\mathcal A} O_{A^-}$ is rigid, its
  index satisfies higher tropical coefficient dynamics.
 \end{proof}

\section{An example in $\mathscr C^3_3$}\label{ex}

In this section, we consider a particular case of the situation studied in \cite{OT}. $A^3_3$ is the
second higher Auslander algebra of linearly oriented $A_3$, and 
$\mathscr C^3_3$ is its $5$-angulated cluster category. We describe them both 
explicitly below.

\begin{example}
Consider the algebra given by the following quiver with relations.
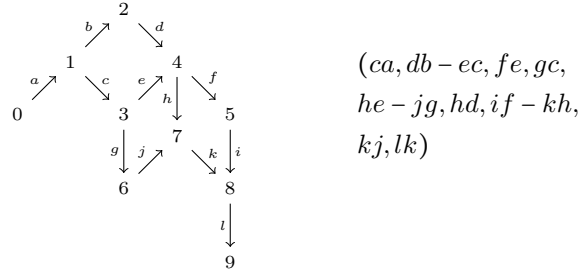
\[ \begin{tikzpicture}[scale=.7,baseline=0pt]
 \node (0) at (0,0) {\(\scriptstyle 0\)};
 \node (1) at (1,1) {\(\scriptstyle 1\)};
 \node (2) at (2,2) {\(\scriptstyle 2\)};
 \node (3) at (2,0) {\(\scriptstyle 3\)};
 \node (4) at (3,1) {\(\scriptstyle 4\)};
 \node (5) at (4,0) {\(\scriptstyle 5\)};
 \node (6) at (2,-1.4) {\(\scriptstyle 6\)};
 \node (7) at (3,-.4) {\(\scriptstyle 7\)};
 \node (8) at (4,-1.4) {\(\scriptstyle 8\)};
 \node (9) at (4,-2.8) {\(\scriptstyle 9\)};
 \draw [->] (0) to node [above left=-2pt] {\(\scriptscriptstyle a\)} (1);
 \draw [->] (1) to node [above left=-2pt] {\(\scriptscriptstyle b\)} (2);
 \draw [->] (1) to node [above right=-2pt] {\(\scriptscriptstyle c\)} (3);
 \draw [->] (2) to node [above right=-2pt] {\(\scriptscriptstyle d\)} (4);
 \draw [->] (3) to node [above left=-2pt] {\(\scriptscriptstyle e\)} (4);
 \draw [->] (4) to node [above right=-2pt] {\(\scriptscriptstyle f\)} (5);
 \draw [->] (3) to node [left=-2pt] {\(\scriptscriptstyle g\)} (6);
 \draw [->] (4) to node [left=-2pt] {\(\scriptscriptstyle h\)} (7);
 \draw [->] (5) to node [right=-2pt] {\(\scriptscriptstyle i\)} (8);
 \draw [->] (6) to node [above left=-2pt] {\(\scriptscriptstyle j\)} (7);
 \draw [->] (7) to node [above right=-2pt] {\(\scriptscriptstyle k\)} (8);
 \draw [->] (8) to node [left=-2pt] {\(\scriptscriptstyle l\)} (9);
\end{tikzpicture} 
\qquad \qquad \begin{aligned} & (ca, db-ec, fe, gc, \\ & he-jg, hd, if-kh, \\ & kj, lk) \end{aligned} \]
This algebra is \( 3 \)-representation finite -- indeed it is a smaller version of \cite[Example~6.13]{I_higherAuslander}.

We will work in the associated \( 5 \)-angulated cluster category \( \mathscr{C} \), as in Definition~\ref{def.cluster_cat}, and illustrate it by drawing our fundamental domain as in Observation~\ref{obs.fundamental_domain}.

We initially consider the cluster tilting object \( T \) which is the image of the algebra in the cluster category. Then the indices of all indecomposable objects are given as in the following picture. (We write the 10-dimensional vectors in two lines purely for space reasons.)

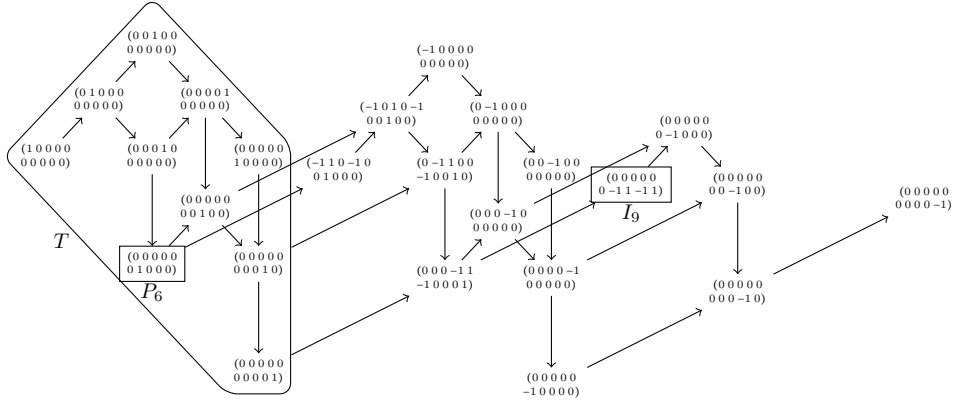
\[ \begin{tikzpicture}[yscale=.75,baseline=0pt,xscale=.70]
 \node [scale=.6] (0) at (0,0) {\( \begin{smallmatrix} ( 1 \; 0 \; 0 \; 0 \; 0 \\ 0 \; 0 \; 0 \; 0 \; 0) \end{smallmatrix} \)};
 \node [scale=.6] (1) at (1,1) {\( \begin{smallmatrix} ( 0 \; 1 \; 0 \; 0 \; 0 \\ 0 \; 0 \; 0 \; 0 \; 0) \end{smallmatrix} \)};
 \node [scale=.6] (2) at (2,2) {\( \begin{smallmatrix} ( 0 \; 0 \; 1 \; 0 \; 0 \\ 0 \; 0 \; 0 \; 0 \; 0) \end{smallmatrix} \)};
 \node [scale=.6] (3) at (2,0) {\( \begin{smallmatrix} ( 0 \; 0 \; 0 \; 1 \; 0 \\ 0 \; 0 \; 0 \; 0 \; 0) \end{smallmatrix} \)};
 \node [scale=.6] (4) at (3,1) {\( \begin{smallmatrix} ( 0 \; 0 \; 0 \; 0 \; 1 \\ 0 \; 0 \; 0 \; 0 \; 0) \end{smallmatrix} \)};
 \node [scale=.6] (5) at (4,0) {\( \begin{smallmatrix} ( 0 \; 0 \; 0 \; 0 \; 0 \\ 1 \; 0 \; 0 \; 0 \; 0) \end{smallmatrix} \)};
 \node [label={[scale=.8,label distance=-1mm]below:\(P_6\)},scale=.6,draw] (6) at (2,-1.9) {\( \begin{smallmatrix} ( 0 \; 0 \; 0 \; 0 \; 0 \\ 0 \; 1 \; 0 \; 0 \; 0) \end{smallmatrix} \)};
 \node [scale=.6] (7) at (3,-.9) {\( \begin{smallmatrix} ( 0 \; 0 \; 0 \; 0 \; 0 \\ 0 \; 0 \; 1 \; 0 \; 0) \end{smallmatrix} \)};
 \node [scale=.6] (8) at (4,-1.9) {\( \begin{smallmatrix} ( 0 \; 0 \; 0 \; 0 \; 0 \\ 0 \; 0 \; 0 \; 1 \; 0) \end{smallmatrix} \)};
 \node [scale=.6] (9) at (4,-3.8) {\( \begin{smallmatrix} ( 0 \; 0 \; 0 \; 0 \; 0 \\ 0 \; 0 \; 0 \; 0 \; 1) \end{smallmatrix} \)};
 
 \node [scale=.6] (10) at (5.5,-0.25) {\( \begin{smallmatrix} ( -1 \; 1 \; 0 \; -1 \; 0 \\ 0 \; 1 \; 0 \; 0 \; 0) \end{smallmatrix} \)};
 \node [scale=.6] (11) at (6.5,0.75) {\( \begin{smallmatrix} ( -1 \; 0 \; 1 \; 0 \; -1 \\ 0 \; 0 \; 1 \; 0 \; 0) \end{smallmatrix} \)};
 \node [scale=.6] (12) at (7.5,1.75) {\( \begin{smallmatrix} ( -1 \; 0 \; 0 \; 0 \; 0 \\ 0 \; 0 \; 0 \; 0 \; 0) \end{smallmatrix} \)};
 \node [scale=.6] (13) at (7.5,-.25) {\( \begin{smallmatrix} ( 0 \; -1 \; 1 \; 0 \; 0 \\ -1 \; 0 \; 0 \; 1 \; 0) \end{smallmatrix} \)};
 \node [scale=.6] (14) at (8.5,0.75) {\( \begin{smallmatrix} ( 0 \; -1 \; 0 \; 0 \; 0 \\ 0 \; 0 \; 0 \; 0 \; 0) \end{smallmatrix} \)};
 \node [scale=.6] (15) at (9.5,-.25) {\( \begin{smallmatrix} ( 0 \; 0 \; -1 \; 0 \; 0 \\ 0 \; 0 \; 0 \; 0 \; 0) \end{smallmatrix} \)};
 \node [scale=.6] (16) at (7.5,-2.15) {\( \begin{smallmatrix} ( 0 \; 0 \; 0 \; -1 \; 1 \\ -1 \; 0 \; 0 \; 0 \; 1) \end{smallmatrix} \)};
 \node [scale=.6] (17) at (8.5,-1.15) {\( \begin{smallmatrix} ( 0 \; 0 \; 0 \; -1 \; 0 \\ 0 \; 0 \; 0 \; 0 \; 0) \end{smallmatrix} \)};
 \node [scale=.6] (18) at (9.5,-2.15) {\( \begin{smallmatrix} ( 0 \; 0 \; 0 \; 0 \; -1 \\ 0 \; 0 \; 0 \; 0 \; 0) \end{smallmatrix} \)};
 \node [scale=.6] (19) at (9.5,-4.05) {\( \begin{smallmatrix} ( 0 \; 0 \; 0 \; 0 \; 0 \\ -1 \; 0 \; 0 \; 0 \; 0) \end{smallmatrix} \)};
 
 \node [label={[scale=.8,label distance=-1mm]below:\(I_9\)},scale=.6,draw] (20) at (11,-.5) {\( \begin{smallmatrix} ( 0 \; 0 \; 0 \; 0 \; 0 \\ 0 \; -1 \; 1 \; -1 \; 1) \end{smallmatrix} \)};
 \node [scale=.6] (21) at (12,.5) {\( \begin{smallmatrix} ( 0 \; 0 \; 0 \; 0 \; 0 \\ 0 \; -1 \; 0 \; 0 \; 0) \end{smallmatrix} \)};
 \node [scale=.6] (23) at (13,-.5) {\( \begin{smallmatrix} ( 0 \; 0 \; 0 \; 0 \; 0 \\ 0 \; 0 \; -1 \; 0 \; 0) \end{smallmatrix} \)};
 \node [scale=.6] (26) at (13,-2.4) {\( \begin{smallmatrix} ( 0 \; 0 \; 0 \; 0 \; 0 \\ 0 \; 0 \; 0 \; -1 \; 0) \end{smallmatrix} \)};
 
 \node [scale=.6] (30) at (16.5,-.75) {\( \begin{smallmatrix} ( 0 \; 0 \; 0 \; 0 \; 0 \\ 0 \; 0 \; 0 \; 0 \; -1) \end{smallmatrix} \)};

 \draw [->] (0) to (1);
 \draw [->] (1) to (2);
 \draw [->] (1) to (3);
 \draw [->] (2) to (4);
 \draw [->] (3) to (4);
 \draw [->] (4) to (5);
 \draw [->] (3) to (6);
 \draw [->] (4) to (7);
 \draw [->] (5) to (8);
 \draw [->] (6) to (7);
 \draw [->] (7) to (8);
 \draw [->] (8) to (9);
 
 \draw [->] (6) to (10);
 \draw [->] (7) to (11);
 \draw [->] (8) to (13);
 \draw [->] (9) to (16);
 
 \draw [->] (10) to (11);
 \draw [->] (11) to (12);
 \draw [->] (11) to (13);
 \draw [->] (12) to (14);
 \draw [->] (13) to (14);
 \draw [->] (14) to (15);
 \draw [->] (13) to (16);
 \draw [->] (14) to (17);
 \draw [->] (15) to (18);
 \draw [->] (16) to (17);
 \draw [->] (17) to (18);
 \draw [->] (18) to (19);
 
 \draw [->] (16) to (20);
 \draw [->] (17) to (21);
 \draw [->] (18) to (23);
 \draw [->] (19) to (26);

 \draw [->] (20) to (21);
 \draw [->] (21) to (23);
 \draw [->] (23) to (26);
 
 \draw [->] (26) to (30);
 
 \draw [rounded corners] (-.8,0) -- (2,2.8) -- (4.6,0.2) -- (4.6,-4.2) -- (3.4,-4.2) -- cycle;
 \node [scale=.8] at (.3,-1.5) {\(T\)};
\end{tikzpicture} \]

The projective \( P_6 \) is mutable, its replacement being the injective \( I_9 \) --- indeed the left mutation 5-angle is given as the image of the exact sequence
\[ 0 \to P_6 \to P_7 \to P_8 \to P_9 \to I_9 \to 0. \]
We denote by \( T^{\star} \) the cluster tilting object obtained by replacing \( P_6 \) by \( I_9 \).

Now we can apply the formula of Corollary~\ref{cor.main_thm} to calculate the index of all indecomposables with respect to \( T^{\star} \).

For those vectors with a positive entry in the 6th coordinate we use the fact that
\[ \index_T( \Sigma I_9) = (-1,1,0,-1,0,0,1,0,0,0). \]
Thus the formula of Corollary~\ref{cor.main_thm} tells us that if the sixth coordinate is $c>0$, we should replace it by $-c$ (which we will also write in the 6th position, but which now represents a coefficient of \( [I_9] \)), and also to add $(c,-c,0,c,0,0,0,0,0,0)$ to the index.

For a vector with a negative entry in the 6th coordinate we recall that
\[ \index_T(I_9) = (0,0,0,0,0,0,-1,1,-1,1). \]
Thus the formula of Corollary~\ref{cor.main_thm} tells us that if the sixth coordinate is $c<0$, we should to change the sign in the 6th coordinate and interpret it as a coefficient of a different basis vector, and also add \((0,0,0,0,0,0,0,c,-c,c)\) to the index.

With these two rules, it is straightforward to calculate all indices:
\[ \begin{tikzpicture}[scale=.75,baseline=0pt]
 \node [scale=.6] (0) at (0,0) {\( \begin{smallmatrix} ( 1 \; 0 \; 0 \; 0 \; 0 \\ 0 \; 0 \; 0 \; 0 \; 0) \end{smallmatrix} \)};
 \node [scale=.6] (1) at (1,1) {\( \begin{smallmatrix} ( 0 \; 1 \; 0 \; 0 \; 0 \\ 0 \; 0 \; 0 \; 0 \; 0) \end{smallmatrix} \)};
 \node [scale=.6] (2) at (2,2) {\( \begin{smallmatrix} ( 0 \; 0 \; 1 \; 0 \; 0 \\ 0 \; 0 \; 0 \; 0 \; 0) \end{smallmatrix} \)};
 \node [scale=.6] (3) at (2,0) {\( \begin{smallmatrix} ( 0 \; 0 \; 0 \; 1 \; 0 \\ 0 \; 0 \; 0 \; 0 \; 0) \end{smallmatrix} \)};
 \node [scale=.6] (4) at (3,1) {\( \begin{smallmatrix} ( 0 \; 0 \; 0 \; 0 \; 1 \\ 0 \; 0 \; 0 \; 0 \; 0) \end{smallmatrix} \)};
 \node [scale=.6] (5) at (4,0) {\( \begin{smallmatrix} ( 0 \; 0 \; 0 \; 0 \; 0 \\ 1 \; 0 \; 0 \; 0 \; 0) \end{smallmatrix} \)};
 \node [label={[scale=.8,label distance=-1mm]below:\(P_6\)},scale=.6,draw] (6) at (2,-1.9) {\(\begin{smallmatrix} ( 1 \; -1 \; 0 \; 1 \; 0 \\ 0 \; -1 \; 0 \; 0 \; 0) \end{smallmatrix} \)};
 \node [scale=.6] (7) at (3,-.9) {\( \begin{smallmatrix} ( 0 \; 0 \; 0 \; 0 \; 0 \\ 0 \; 0 \; 1 \; 0 \; 0) \end{smallmatrix} \)};
 \node [scale=.6] (8) at (4,-1.9) {\( \begin{smallmatrix} ( 0 \; 0 \; 0 \; 0 \; 0 \\ 0 \; 0 \; 0 \; 1 \; 0) \end{smallmatrix} \)};
 \node [scale=.6] (9) at (4,-3.8) {\( \begin{smallmatrix} ( 0 \; 0 \; 0 \; 0 \; 0 \\ 0 \; 0 \; 0 \; 0 \; 1) \end{smallmatrix} \)};
 
 \node [scale=.6] (10) at (6,0) {\( \begin{smallmatrix} ( 0 \; 0 \; 0 \; 0 \; 0 \\ 0 \; -1 \; 0 \; 0 \; 0) \end{smallmatrix} \)};
 \node [scale=.6] (11) at (7,1) {\( \begin{smallmatrix} ( -1 \; 0 \; 1 \; 0 \; -1 \\ 0 \; 0 \; 1 \; 0 \; 0) \end{smallmatrix} \)};
 \node [scale=.6] (12) at (8,2) {\( \begin{smallmatrix} ( -1 \; 0 \; 0 \; 0 \; 0 \\ 0 \; 0 \; 0 \; 0 \; 0) \end{smallmatrix} \)};
 \node [scale=.6] (13) at (8,0) {\( \begin{smallmatrix} ( 0 \; -1 \; 1 \; 0 \; 0 \\ -1 \; 0 \; 0 \; 1 \; 0) \end{smallmatrix} \)};
 \node [scale=.6] (14) at (9,1) {\( \begin{smallmatrix} ( 0 \; -1 \; 0 \; 0 \; 0 \\ 0 \; 0 \; 0 \; 0 \; 0) \end{smallmatrix} \)};
 \node [scale=.6] (15) at (10,0) {\( \begin{smallmatrix} ( 0 \; 0 \; -1 \; 0 \; 0 \\ 0 \; 0 \; 0 \; 0 \; 0) \end{smallmatrix} \)};
 \node [scale=.6] (16) at (8,-1.9) {\( \begin{smallmatrix} ( 0 \; 0 \; 0 \; -1 \; 1 \\ -1 \; 0 \; 0 \; 0 \; 1) \end{smallmatrix} \)};
 \node [scale=.6] (17) at (9,-.9) {\( \begin{smallmatrix} ( 0 \; 0 \; 0 \; -1 \; 0 \\ 0 \; 0 \; 0 \; 0 \; 0) \end{smallmatrix} \)};
 \node [scale=.6] (18) at (10,-1.9) {\( \begin{smallmatrix} ( 0 \; 0 \; 0 \; 0 \; -1 \\ 0 \; 0 \; 0 \; 0 \; 0) \end{smallmatrix} \)};
 \node [scale=.6] (19) at (10,-3.8) {\( \begin{smallmatrix} ( 0 \; 0 \; 0 \; 0 \; 0 \\ -1 \; 0 \; 0 \; 0 \; 0) \end{smallmatrix} \)};
 
 \node [label={[scale=.8,label distance=-1mm]below:\(I_9\)},scale=.6,draw] (20) at (12,0) {\( \begin{smallmatrix} ( 0 \; 0 \; 0 \; 0 \; 0 \\ 0 \; 1 \; 0 \; 0 \; 0) \end{smallmatrix} \)};
 \node [scale=.6] (21) at (13,1) {\( \begin{smallmatrix} ( 0 \; 0 \; 0 \; 0 \; 0 \\ 0 \; 1 \; -1 \; 1 \; -1) \end{smallmatrix} \)};
 \node [scale=.6] (23) at (14,0) {\( \begin{smallmatrix} ( 0 \; 0 \; 0 \; 0 \; 0 \\ 0 \; 0 \; -1 \; 0 \; 0) \end{smallmatrix} \)};
 \node [scale=.6] (26) at (14,-1.9) {\( \begin{smallmatrix} ( 0 \; 0 \; 0 \; 0 \; 0 \\ 0 \; 0 \; 0 \; -1 \; 0) \end{smallmatrix} \)};
 
 \node [scale=.6] (30) at (16,0) {\( \begin{smallmatrix} ( 0 \; 0 \; 0 \; 0 \; 0 \\ 0 \; 0 \; 0 \; 0 \; -1) \end{smallmatrix} \)};

 \draw [->] (0) to (1);
 \draw [->] (1) to (2);
 \draw [->] (1) to (3);
 \draw [->] (2) to (4);
 \draw [->] (3) to (4);
 \draw [->] (4) to (5);
 \draw [->] (3) to (6);
 \draw [->] (4) to (7);
 \draw [->] (5) to (8);
 \draw [->] (6) to (7);
 \draw [->] (7) to (8);
 \draw [->] (8) to (9);
 
 \draw [->] (6) to (10);
 \draw [->] (7) to (11);
 \draw [->] (8) to (13);
 \draw [->] (9) to (16);
 
 \draw [->] (10) to (11);
 \draw [->] (11) to (12);
 \draw [->] (11) to (13);
 \draw [->] (12) to (14);
 \draw [->] (13) to (14);
 \draw [->] (14) to (15);
 \draw [->] (13) to (16);
 \draw [->] (14) to (17);
 \draw [->] (15) to (18);
 \draw [->] (16) to (17);
 \draw [->] (17) to (18);
 \draw [->] (18) to (19);
 
 \draw [->] (16) to (20);
 \draw [->] (17) to (21);
 \draw [->] (18) to (23);
 \draw [->] (19) to (26);

 \draw [->] (20) to (21);
 \draw [->] (21) to (23);
 \draw [->] (23) to (26);
 
 \draw [->] (26) to (30);
 \draw [rounded corners] (-.6,0.2) -- (2,2.8) -- (10.7,2.8) -- (12.7,0.4) -- (12.7,-.7) -- (11.3,-.7) -- (8.7,2.5) -- (4.6,2.5) -- (4.6,-4.2) -- (3.4,-4.2) -- (3.4,-2.2) -- (1.6,-.4) -- (-.6,-.4) -- cycle;
 \node [scale=.8] at (4.9,1.6) {\(T^{\star}\)};
\end{tikzpicture} \]
\end{example}

\end{document}